\documentclass[aap]{imsart}

\RequirePackage{amsthm,amsmath,amsfonts,amssymb}
\RequirePackage[numbers]{natbib}

\usepackage{amsmath,amsthm,amsfonts,amssymb,mathrsfs,bm,bbm}
\usepackage[usenames]{color}
\usepackage{hyperref}
\usepackage{amssymb}
\usepackage{appendix}

\startlocaldefs

\newtheorem{theorem}{Theorem}[section]
\newtheorem{definition}[theorem]{Definition}
\newtheorem{proposition}[theorem]{Proposition}
\newtheorem{corollary}[theorem]{Corollary}
\newtheorem{lemma}[theorem]{Lemma}
\newtheorem{remark}[theorem]{Remark}
\newtheorem{example}[theorem]{Example}

\newtheorem{prop}[theorem]{Proposition}

\newtheorem{lem}[theorem]{Lemma}

\newcommand{\supp}{{\rm supp}}

\newcommand{\Spec}{{\rm Spec}}

\newcommand{\del}{\partial}

\renewcommand{\Im}{\mathop{\mathrm{Im}}}
\renewcommand{\Re}{\mathop{\mathrm{Re}}}

\newcommand{\Var}{{\mathop{\mathrm{Var}}\nolimits}}

\newcommand{\ind}{{\bf 1}}

\newcommand{\D}{\mathbb{D}}

\newcommand{\N}{\mathbb{N}}
\newcommand{\Z}{\mathbb{Z}}
\newcommand{\R}{\mathbb{R}}

\renewcommand\P{\mathbb{P}}

\newcommand{\E}{\mathbb{E}}

\newtheorem{model}{Model}[section]

\def\Z{\mathbb{Z}}
\def\R{\mathbb{R}}
\def\C{\mathbb{C}}
\def\N{\mathbb{N}}
\def\E{\mathbb{E}}
\def\P{\mathbb{P}}

\def\b{\beta}
\def\a{\alpha}

\def\s{\sigma}

\def\ph{\varphi}
\def\var{\mathrm{Var}}

\def\sin{\mathrm{sin}}

\def\L{\Lambda}
\def\la{\lambda}

\def\del{\delta}
\def\Int{\mathrm{Int}}

\def\d{\mathrm{d}}
\def\det{\mathrm{Det}}
\def\op{\mathrm{op}}
\def\ind{\mathbbm{1}}
\def\Tr{\mathrm{Tr}}
\def\Vol{\mathrm{Vol}}

\def\bs{{\bm{\sigma}}}

\def\bp{\overline{\varphi}}
\def\hp{\hat{\varphi}}
\def\ph{\varphi}
\def\cb{\chi_B}
\def\cbl{\chi_B^L}

\def\ol{\overline}

\def\Det{\operatorname{Det}}
\def\Tr{\operatorname{Tr}}
\def\Vol{\operatorname{Vol}}
\def\op{{\mathrm{op}}}

\makeatletter
\newcommand{\addresseshere}{%
  \enddoc@text\let\enddoc@text\relax
}
\makeatother

\newcommand{\numberthis}{\addtocounter{equation}{1}\tag{\theequation}}
\renewcommand{\l}[0]{\left }
\renewcommand{\r}[0]{\right}

\makeatletter
\renewcommand*{\@cite@ofmt}{\hbox}
\makeatother

\endlocaldefs

\begin{document}

\begin{frontmatter}

\title{Gaussian fluctuations for spin systems and point processes: near-optimal rates via quantitative Marcinkiewicz's theorem}

\runtitle{Gaussian fluctuations for spin systems and point processes}

\begin{aug}

\author[A]{\fnms{Tien-Cuong}~\snm{Dinh}\ead[label=e1]{matdtc@nus.edu.sg}},
\author[A]{\fnms{Subhroshekhar}~\snm{Ghosh}\ead[label=e2]{subhrowork@gmail.com}},
\author[A]{\fnms{Hoang-Son}~\snm{Tran}\ead[label=e3]{hoangson.tran@u.nus.edu}}
\and
\author[A]{\fnms{Manh-Hung}~\snm{Tran}\ead[label=e4]{e0511873@u.nus.edu}}

\address[A]{Department of Mathematics,  National University of Singapore - 10, Lower Kent Ridge Road - Singapore 119076 
\printead[presep={,\ }]{e1,e2,e3,e4}}

\end{aug}

\begin{abstract}

In this work, we establish asymptotically Gaussian fluctuations for functionals of a large class of spin models and strongly correlated random point fields, achieving near-optimal rates.
For spin models, we demonstrate Gaussian asymptotics for the magnetization (i.e., the total spin) for a wide class of ferromagnetic spin systems on Euclidean lattices, in particular those with continuous spins. Specific applications include, in particular, the celebrated XY and Heisenberg models under ferromagnetic conditions, and more broadly, systems with very general rotationally invariant spins in arbitrary dimensions. We address both the setting of free boundary conditions and a large class of ferromagnetic boundary conditions, and our CLTs are endowed with near-optimal rate of $O(\log |\L| \cdot |\L|^{-1/2})$ in the Kolmogorov-Smirnov distance, where the system size is $|\L|$. Our approach leverages the classical Lee-Yang theory for the zeros of partition functions, and subsumes as a special case results of Lebowitz, Ruelle, Pittel and Speer on CLTs in discrete statistical mechanical models for which we obtain sharper convergence rates. In a different direction, we obtain CLTs for linear statistics of a wide class of point processes known as $\alpha$-determinantal point processes which interpolate between negatively and positively associated random point fields (including the usual determinantal, permanental and Poisson point  processes). We contribute a unified approach to CLTs in such models (agnostic to the parameter $\alpha$ that modulates the nature of association). Our methods are able to address a broad class of kernels including in particular those with slow spatial decay (such as the Bessel kernel in general dimensions). Significantly, our approach is able to analyse such processes in dimensions $\ge 3$, where structural alternatives such as connections to random matrix theory are not available, and obtain explicit rates for fast convergence in a wide spectrum of models. A key ingredient of our approach is a broad, quantitative extension of the classical Marcinkiewicz Theorem that holds under the limited condition that  the characteristic function is non-vanishing \textit{only on a bounded disk}. This technique  complements classical work of Ostrovskii, Linnik, Zimogljad and others, as well as recent work of Michelen and Sahasrabudhe, and Eremenko and Fryntov. Our applications demonstrate the crucial significance of it being sufficient to control the characteristic function \textit{only} on a small (and possibly shrinking) disk near the origin, a feature that is available in our approach (in contrast to other Marcinkiewicz-type theorems). In spite of the general applicability of the results, our rates for the CLT match the classic Berry-Esseen bounds for independent sums up to a log factor.

\end{abstract}



\end{frontmatter}

\newpage
\tableofcontents


\section{Introduction}
\label{s:Intro}

\subsection{Gaussian fluctuations for strongly dependent random systems} \label{s:applications}

The theory of Gaussian fluctuations has a long history in statistical mechanics and probability theory. A key ingredient for most classical techniques to establish such central limit theorems is to express the random variable of interest as a sum of small,  independent components. While the classical CLT for triangular arrays of random variables explicitly relies on such criteria, the more dependent nature of the models in statistical mechanics necessitates a search for approximate independence of some sort, as a substitute for the traditional independent setting. As such, most of the existing methods for obtaining CLTs depends on independence in a direct or an indirect manner, and  demonstrating Gaussian fluctuations for strongly dependent statistical mechanical systems is considered to be a challenge. In the strongly dependent settings in which Gaussian fluctuations are known, such as \cite{CL,Leb,GL-1,GL-2,Sos}, the techniques often hinge crucially on specific properties of the system, for instance specific cancellation properties originating from the particular combinatorial structure of the model.



In this article, we derive central limit theorems for observables of natural interest in two distinct types of statistical mechanical models, using a unified approach that is potentially applicable to many other strongly dependent random systems. On one hand, we obtain CLTs for the magnetization (i.e., the total spin) in very general ferromagnetic spin systems (including, in particular, novel Gaussian fluctuation results for continuous spin systems such as the XY model and the Heisenberg ferromagnet). On the other hand, we demonstrate Gaussian fluctuations for linear statistics for a very general class of point  processes, which includes (and in fact, interpolates between) attractive (Bosonic) and repuslive (Fermionic) processes. A leitmotif of these stochastic systems is their strongly correlated nature, which renders ineffective most of the common approaches to CLT that involve exploiting independence or approximate independence in some form. Our CLTs are based on an effective, quantitative version of the classical Marcinkiewicz theorem, which entails non-vanishing of the characteristic function \emph{only on a bounded disk}, which is traded off against the growth of the characteristic function on the same disk.
In spite of the general scope of the method, the CLTs we obtain achieve near-optimal rates, often bounded above (up to logarithmic factors) by the inverse square root of the system size, which is the classical Berry-Esseen rate for sums of independent random variables. 

Our approach raises interesting possibilities for investigating the large scale stochastic properties of strongly dependent random systems; for instance, concentration phenomena for functionals defined on such models. On this point, an analogy may be drawn with other traditional techniques for studying asymptotic fluctuations. In particular, in \cite{Cha1,Cha2}, it was demonstrated that the Stein's method, a classical tool for obtaining CLTs, could be extended and modified to establish concentration inequalities for a large class of dependent stochastic models. In a similar vein, we envisage enhancements of our approach to yield other significant properties (such as concentration of measure) for strongly dependent random systems,  in particular continuous spin models and random point fields. This is left as a natural avenue for further investigation, pursuant to the results obtained in this paper.

\subsection{Spin systems and Lee-Yang theory} \label{s:appl_spin}


In the context of spin systems, the fluctuation theory for fundamental observables (such as the total spin) is perhaps best understood in the setting of classical Ising models with $\pm$ spins \cite{Ellis_book}. Broadly speaking, it is understood that under ferromagnetic conditions, the total spin satisfies a CLT after appropriate centering and scaling, whereas in the absence of such conditions, the asymptotic fluctuations may be non-Gaussian. However, most of the known techniques for obtaining such results relies on the discrete nature of the spins, which allows one to make use tools from the theory of combinatorial stochastic processes. 

Models with continuous spins are, however, of fundamental importance in statistical mechanics. These include the important spin systems with continuous or vector valued spins; in particular, the XY model (i.e., the plane rotator model) with $\mathbb{S}^1$-valued spins and the classical Heisenberg spin system with $\mathbb{S}^2$-valued spins. 
In this work, we obtain CLTs for general, possibly continuous-valued and multi-component spin distributions.
We emphasize that our technique is independent of the discrete structure of the spins, and obtain novel CLTs in particular for the magnetization in the XY model as well as the classical Heisenberg  model under ferromagnetic conditions. In fact, even for one-component spins, our results hold for a very general class of continuous spin systems that appears to be beyond the scope of existing literature. 

In the setting of spin systems, our approach leverages the classical Lee-Yang theory for the zeros of partition functions; for a detailed and modern exposition to Lee-Yang theory, we refer the reader to \cite{FR}; for the original works of Lee and Yang see \cite{YL-1,LY-2}.
In particular, our approach subsumes as a special case a technique of Lebowitz, Ruelle, Pittel and Speer for deriving CLTs in discrete statistical mechanical models \cite{LPRS}, which develops on an earlier Lee-Yang theory based approach in the physical paper \cite{IS}. Restricted to the discrete setting, we obtain sharper convergence rates  (for instance, we obtain a rate of $O(\log N \cdot N^{-1/2})$ compared to $O(N^{-1/6})$ as in \cite{LPRS} Theorem 2.1; whereas the classical Berry-Esseen bound is $O(N^{-1/2})$).

To lay out our results more precisely, consider the lattice $\Z^d$ and let $\L \subset \Z^d$ be a $d$-dimensional cube. For any two neighbouring vertices $x,y \in \Z^d$, we write  $(x,y)$ for the edge connecting $x$ and $y$. We identify $\L$ with the graph formed by nearest-neighbor pairs of vertices of $\L$. We also denote by $\partial \L$ the \emph{boundary} of $\L$, i.e., the set of all vertices in $\Z^d \setminus \L$ that are connected to some vertices in $\L$.  
A \emph{spin configuration} on $\L$, by definition, is a map $\bs_\L : \L \rightarrow \R^N $. For each site $x\in \L$, we denote the spin at $x$ by $\bs_x:= \bs_\L(x) = (\s_x^1,\ldots,\s_x^N) \in \R^N $.

\begin{remark} \label{rem:abuse_notn}
Although we denote spins as well as standard deviation by the same symbol $\s$, the connotation of $\s$ should be understood from the context.
\end{remark}

Let $\mu_0$ be a probability distribution on $\R^N$. Our spin systems are defined as follows:

\begin{definition}
(i) (Spin systems with free boundary condition) Consider the Hamiltonian
\begin{equation} \label{eq:Hamiltonian}
H_\L(\bs_\L) :=-\sum_{(x,y) \subset \L} \sum_{i=1}^N J_{xy}^i\s_x^i\s_y^i - \sum_{x \in \L} \sum_{i=1}^N h_x^i \s_x^i ,
\end{equation}
where $ J_{xy}^i=J_{yx}^i$ are coupling constants for the edge $(x,y)$ in the direction $i \in \{1,\ldots,N\}$, and $\bm h_x := (h_x^1,\ldots,h_x^N)$ is the external magnetic field at $x \in \L$. 
A spin system with free boundary condition on $\L$ is a random spin configuration $\bs_\L$, which is distributed according to
the following Gibbs distribution on $(\R^N)^{\otimes |\L|}$
\begin{equation} \label{eq:measure}
\P_\L(\bs_\L) \propto \exp(- \b H_\L(\bs_\L)) \prod_{x \in \L} \d \mu_0(\bs_x),
\end{equation}
where $\b>0$ is a parameter, called the inverse temperature.

\medskip
\noindent
(ii) (Spin systems with boundary condition) Let $\mathcal V= \{\bm v_x \in \R^N: x\in \partial \L\}$ be a collection of given vectors (boundary condition) and let $\L':= \L \cup \partial \L$. Consider the Hamiltonian
\begin{equation} \label{eq:cond-Hamiltonian}
H_\L^{\mathcal V}(\bs_\L) :=-\sum_{(x,y) \subset \L'} \sum_{i=1}^N J_{xy}^i\s_x^i\s_y^i - \sum_{x \in \L} \sum_{i=1}^N h_x^i \s_x^i ,
\end{equation}
with the convention that $\bs_x = \bm v_x$ for every $x\in \partial \L$.
A spin system with boundary condition $\mathcal V$ on $\L$ is a random spin configuration $\bs_\L$, which is distributed according to
the following Gibbs distribution on $(\R^N)^{\otimes |\L|}$
\begin{equation} \label{eq:cond-measure}
\P_\L^{\mathcal V}(\bs_\L) \propto \exp(- \b H_\L^{\mathcal V}(\bs_\L)) \prod_{x \in \L} \d \mu_0(\bs_x),
\end{equation}
where $\b>0$ is the inverse temperature.
\end{definition}


For each $\del \in \R$, we define the following subset of $\C^N$
\begin{equation} \label{eq:goodvector}
\Omega^N_\del := \Big \{ \bm z = (z_1,\ldots,z_N) \in \C^N : \Re z_1 - \sum_{j=2}^N |z_j| \ge \del \Big \}.
\end{equation}
For the purposes of our CLT, we will work in the following modelling setup.

\begin{model}[Model classes for CLTs in Ferromagnetic spin systems] \label{model:spin}

We have one among the following choices of  $\mu_0$ and interactions $J^i_{xy}$. The coupling constants $J_{xy}^i$ and the external magnetic field $\bm h_x$ are always assumed to be uniformly bounded. Assume further that there exists $\delta>0$ (not depending on $\L$) such that $\bm h_x \in \Omega^N_\del$ for every $x \in \L$.
\begin{itemize}
\item [(a)] $N=1$: $\mu_0$ is a  compactly supported even measure on $\R$ (that is not degenerate at $0$) satisfying
$$\int_\R e^{u\s} \d \mu_0(\s) \neq 0 ~\text{ whenever } \Re u>0$$
and the ferromagnetic condition $J_{xy} \ge 0$ is satisfied for all edges $(x,y)$. 
\item [(b)] $N=2$ (XY model): $\mu_0$ is the uniform measure on the unit circle $\mathbb{S}^1$, and  the ferromagnetic condition $J_{xy}^1 \ge |J_{xy}^2|$ is satisfied for all edges $(x,y)$.
\item [(c)] $N=3$ (Heisenberg model): $\mu_0$ is the uniform measure on the unit sphere $\mathbb{S}^2$, and  the ferromagnetic condition $J_{xy}^1 \ge \max\{|J_{xy}^2|,|J_{xy}^3|\}$ is satisfied for all edges $(x,y)$. 
\item [(d)]  $N\geq 2$:  $\mu_0$ is rotationally invariant, compactly supported and satisfies 
$$\int_{\R^N} e^{u\s^1} \d\mu_0(\bs) \neq 0 ~ \text{ whenever } \Re u \neq 0$$
and the couplings satisfy
$$J^1_{xy} \geq \sum_{i=2}^N |J^i_{xy}|, \quad \text{for all edges } (x,y).$$
	\end{itemize}
\end{model}

Let $\bm S_\L := \sum_{x\in \L} \bs_x$ be the total spin of the system, then we may state:
\begin{theorem}[CLTs for spin systems with free boundary conditions] \label{thm:CLT_spin}
Let $\bs_\L$ be a  spin system  defined on a cube $\L \subset \Z^d$ with free boundary condition, satisfying the hypotheses in Model \ref{model:spin}. Then $\bm S_\L$ satisfies a CLT as $\L \uparrow \Z^d$. Moreover, we have the quantitative estimate
$$\sup_{\bm \alpha \in \mathbb S^{N-1}} \d_{KS}\Big ({\langle \bm S_\L, \bm \alpha \rangle - \E \langle \bm S_\L, \bm \alpha \rangle \over \Var[\langle \bm S_\L, \bm \alpha \rangle ]^{1/2}}, N(0,1) \Big ) \le C \cdot {\log |\L| \over |\L|^{1/2}},$$
where $\d_{KS}(\cdot,\cdot)$ denotes the Kolmogorov-Smirnov distance between two probability distributions, and $C$ is a positive constant.
\end{theorem}

In parallel to Theorem \ref{thm:CLT_spin}, we also demonstrate a CLT for spin systems with \textit{ferromagnetic} boundary conditions.

\begin{theorem}[CLTs for spin systems with ferromagnetic boundary conditions] \label{thm:CLT_spinb}
Let $\bs_\L$ be a spin system defined on a cube $\L \subset \Z^d$ with ferromagnetic boundary condition $\mathcal V \subset \Omega^N_0$, satisfying the hypotheses in Model \ref{model:spin}. 
Then $\bm S_\L$ satisfies a CLT as $\L \uparrow \Z^d$. Moreover, we have the quantitative estimate
$$\sup_{\bm \alpha \in \mathbb S^{N-1}} \d_{KS}\Big ({\langle \bm S_\L, \bm \alpha \rangle - \E \langle \bm S_\L, \bm \alpha \rangle \over \Var[\langle \bm S_\L, \bm \alpha \rangle ]^{1/2}}, N(0,1) \Big ) \le C \cdot {\log |\L| \over |\L|^{1/2}},$$
where $\d_{KS}(\cdot,\cdot)$ denotes the Kolmogorov-Smirnov distance between two probability distributions, and $C$ is a positive constant.
\end{theorem}

A key feature of Theorem \ref{thm:CLT_spin} is that it provides a unified approach that may be applied irrespective of the number of components $N$ of the spins. This is particularly significant in the context of the fact that many effective tools to understand behaviour of spin systems, become unavailable with increasing number of components $N$, and rigorous analysis is often possible only in high temperature regimes via perturbative expansions. Furthermore, in the few cases where CLT is known (principally, the classical Ising model \cite{Ellis_book}), the literature on convergence rates is very limited. In the multi-component models and the generalised Ising models considered in this article, to our knowledge, CLT for the total spin is not known. In the one-component case (a) in Model \ref{model:spin}, our approach may be applied to the setting where the spins are not compactly supported but only have a finite exponential moment; however, for the sake of brevity and unity of presentation, we adhere to the case of compactly supported spins in this article. On this note, we observe that a natural extension of our approach can be considered in the setting of systems with ambient disorder, for which central limit theorems for the free energy has attracted interest in the literature (see, eg, 
\cite{Cha3} for random field Ising models, and \cite{Lam-Sen} for disordered monomer-dimer systems). 


\subsection{Linear statistics of $\a$-determinantal processes} \label{s:appl_det}

Determinantal point processes \cite{Sos,Bor,HKPV} have emerged as a significant probabilistic model for capturing a wide class of phenomena in statistical physics, quantum theory, combinatorics, representation theory and the theory of integrable systems. These processes are characterised by their so-called correlation functions, which represent the probability (densities) of having particles at specified locations; the latter being given by a determinants of a certain kernel matrix with respect to a background measure.

Determinantal processes have been extended and generalised in multiple directions; these include in particular the so-called permanental processes where the determinantal structure of the correlation functions is replaced by a permanental one.  A standard one-parameter family of point processes that interpolate between the determinantal and permanental ones, also including the classical Poisson point process, is the family of $\a$-determinantal processes. To define these processes, we first define the notion of the $\a$-determinant of a matrix $A \in \C^{n \times n}$
\begin{equation} \label{eq:adet}
\det_\a[A] :=\sum_{\s \in S_n} \a^{n-\nu(\s)} \prod_{i=1}^n A_{i\s(i)},
\end{equation}
where $S_n$ is the symmetric group on $n$ symbols, and $\nu(\s)$ stands for the number of cycles in $\s \in S_n$. It is easy to see that for $\a=-1$, $\det_\a$ is the usual determinant; for $\a=+1$, it is the permanent and for $\a=0$, we have $\det_0=\prod_{i=1}^n A_{ii}$.

Let $\Xi$ be a locally compact Polish space endowed with a non-negative Borel measure $\mu$ and a Hermitian kernel $K : \Xi \times \Xi \rightarrow \C$.
The $\a$-determinantal point process on $\Xi$ with kernel $K$ and background measure $\mu$ is a random locally finite point set on $\Xi$ such that for any finite subset $\{x_1,\ldots,x_n\} \subset \Xi$, the probability (density, with respect to $\mu^{\otimes n}$) of having points at these locations is given by the $n$-point correlation function
\[ \rho_n(x_1,\ldots,x_n)=\det_\a \l[(K(x_i,x_j))_{1 \le i,j \le n} \r]. \]

It is known that for $\a>0$, $\a$-determinantal processes exist under the conditions that $K$ is a bounded symmetric integral operator on $L^2(
\mu)$ that is positive semi-definite and locally trace class, and $\a \in \{ 2/m : m \in \N_+ \}$. For $\a<0$, it is additionally necessary that $\text{Spec}(K) \subset [0,-1/\alpha]=[0,1/|\alpha|]$. For a detailed account of  $\a$-determinantal processes, we refer the reader to \cite{ST,Maj,Bac}. Clearly, $\a=-1$ corresponds to determinantal point processes, whereas $\a=+1$ corresponds to permanental processes.

\subsubsection{CLTs for linear statistics}

Let $X$ be a point process on a space $\Xi$. For a test function $\varphi : \Xi \rightarrow \R$ with compact support, the linear statistic $\L(\varphi)$ is given by the sum $\L(\varphi)=\sum_{x \in X} \varphi(x)$. Linear statistics are fundamental objects of interest in understanding point processes; indeed, under very general conditions, the statistical law of a point process is completely determined by the distribution of its linear statistics.
In the present work, we consider $\Xi=\R^d$ and $\varphi_L(\cdot):=\varphi(\cdot/L)$, for a parameter $L>0$. We will investigate the family of random variables given by the linear statistics $\L(\varphi_L)$ and, under very general  conditions, obtain a CLT for them  as $L \to \infty$.

CLTs for such families of linear statistics of $\a$-determinantal processes are known in  specialised situations \cite{ST}. The fundamental problem in CLTs for linear statistics is that, although $\L(\varphi_L)$ is expressible as a sum, the summands are highly correlated because of the correlation structure of the $\a$-determinantal point process, and hence standard techniques for CLTs for sums of independent variates cannot be applied in this setting. Generally speaking, such CLTs require a delicate and elaborate analysis of cumulant expansions of $\L(\varphi_L)$, often exploiting particular analytical structures accorded by the specific setting under consideration. For instance, \cite{ST} deals with the scenario where $K$ is a translation invariant convolutional operator and $\mu$ is the Lebesgue measure on $\R^d$. Such invariance assumption allows the application of Fourier analytic techniques, which are exploited to perform asymptotic analysis of cumulant expansions for $\L(\varphi_L)$. Another special case where CLTs are better understood is that of determinantal processes, where, using a variety of specialised tools such as   Fourier analysis and orthogonal polynomials, progress has been achieved in different settings \cite{ST,Sos-1,Duits,Bar,RV}. Furthermore, in most cases, the literature appears to be limited regarding rates of convergence to normality.

We are able to invoke the techniques in this article to provide a succinct and self contained proof of CLTs for $(\L(\varphi_L))_{L>0}$ for general $\a$-determinantal processes,  along with explicit rates of convergence.

\begin{theorem} \label{thm:CLT_det}
Let  $X$ be an $\a$-determinantal process on $\R^d$ with a Hermitian kernel $K$ and background measure $\mu$ satisfying
\begin{enumerate}
\item[(i)] $K$ is a bounded symmetric integral operator on $L^2(\mu)$ which is locally trace class and positive semi-definite;
\item[(ii)] $\alpha \in \{{2\over m}: m \in \N_+\}\cup \{-{1\over m}: m \in \N_+ \}$, and for $\alpha < 0$ we further require that $I+\alpha K$ is also positive semi-definite, i.e., $\Spec (K) \subset [0,-1/\alpha]$;
\item[(iii)] There exists $C>0$ such that $\E[X(B)] \le C \Vol(B)$ for every compact subsets $B$ of $\R^d$, where $X(B)$ denotes the number of points of $X$ inside $B$.
\end{enumerate}
Let $\varphi : \R^d \rightarrow \R$ be a bounded, compactly supported function and $\varphi_L(\cdot) := \varphi( \cdot / L), L>0$. If 
\begin{equation} \label{eq:variancegrowth}
{\frac{\var[\L(\varphi_L)]^{1/2}} {\log L}} \to \infty \quad \text{ as }L\rightarrow \infty, 
\end{equation}
then $\{\L(\varphi_L)\}_{L>0}$ satisfies a CLT upon centering and scaling as $L \uparrow \infty$, with its Kolmogorov-Smirnov distance  from a standard Gaussian decaying at the rate $O(\log L \cdot (\var[\L(\ph_L)])^{-1/2})$.
\end{theorem}

The assumptions in Theorem \ref{thm:CLT_det} are actually mild and reasonable. Indeed, conditions (i) and (ii) ensure the existence of the $\a$-determinantal point process $X$ with the prescribed kernel $K$ and background measure $\mu$ (see \cite{ST}); whereas the condition (iii) is easily satisfied for most of point processes of interest. For example, if $\mu$ has a bounded density with respect to the Lebesgue measure on $\R^d$, i.e., $\d\mu(x) = f(x)\d x$ with $f(x) \le c, \forall x\in \R^d$ for some constant $c>0$ and the first intensity $K(x,x) \le b, \forall x\in \R^d$ for some $b>0$, then for any compact set $B\subset \R^d$ 
$$\E[X(B)] = \int_{B} K(x,x)f(x)\d x \lesssim \Vol (B).$$
However, to obtain CLTs, we still need to verify a condition \eqref{eq:variancegrowth} on the growth of variance of linear statistics. In general, this condition will require extra assumptions on the kernel and the background measure. In the next section, we provide  some sufficient conditions for variance growth of linear statistics.

\subsubsection{Sufficient conditions for variance growth}
We lay out the modelling setup for which the growth for variance can be verifed. In what follows, we will use the notation $C(x,r)$ to denote the cube with centre $x \in \R^d$ and side length $r$,   and $A_n^x(r)$ to denote the annulus 
$$A_n^x(r):=\{y \in \R^d : nr \le \|x-y\| \le (n+1)r\}.$$

\begin{model} \label{model:det}
Let $\mu$ be such that $\d \mu(x) = f(x) \d x$ with $c^{-1} \le f(x) \le c, ~\forall x \in \R^d$ for some constant $c>0$, and $K(x,x) \le b$ for all $x \in \R^d$.  
Let $\varphi : \R^d \rightarrow \R$ be a bounded compactly supported function with $\|\varphi\|_2 >0$. 
We additionally have one of the following conditions hold:
\begin{itemize}
\item[(i)] For any $\del>0$, there exist $m_\del >0$ and $r=r(\del)>0$ such that
$$\Vol(\{y \in C(x,r) : K(y,y) \ge m_\del \}) \ge (1-\del) \Vol(C(x,r)), \quad\forall x \in \R^d.$$
If $\alpha < 0$, we further require that $\|K\|_{\op} < |\alpha|^{-1}$.
\item[(ii)] $\a<0$,  and either of :
         \begin{itemize}
         \item[(ii.a)] $|K(x,y)| \ge a, ~\forall \|x-y\| < \del$ for some $a,\del>0$ and $d>4$.
         \item[(ii.b)] there exist $\b \in ({d\over 2}, d )$, positive constants $r,c_1,c_2 \in \R_+$ and $n_0 \in \N_+$, such that for every $x \in \R^d$, the set
\[ E_x := \{y : |K(x,y)| \ge c_1 \|x-y\|^{-\b} \} \]
satisfies 
\[\Vol(E_x \cap A_n^x(r)) \ge c_2 \Vol(A_n^x(r)),\quad \forall n \ge n_0.\]
         \item[(ii.c)] $d>2$, and
         \begin{equation} \label{eq:3moment}
         \sup_{x\in \R^d} \int_{\R^d} \|x-y\|^3 |K(x,y)|^2 \d y <\infty
         \end{equation}
         and
         \begin{equation} \label{eq:rot_inv}
         \inf_{x\in\R^d} \inf_{u\in \mathbb S^{d-1}} \int_{\R^d} |\langle y-x,u\rangle |^2 |K(x,y)|^2 \d y > 0,
         \end{equation}
         and we further assume $\varphi \in C^2$.
         \end{itemize}
\end{itemize}
\end{model}

Under such conditions, we are able to show:
\begin{theorem} \label{thm:var_det}
Let $X$ be an $\alpha$-determinantal process on $\R^d$ with a positive semi-definite kernel $K$ and back ground measure $\mu$ on $\R^d$ and $\varphi:\R^d \rightarrow \R$ be a test function as in the Model \ref{model:det}. Then
$$\var[\L(\varphi_L)]\gtrsim L^{\eta},$$
 where 
\begin{enumerate}
\item[(i)] $\eta = d$ for the model class Model \ref{model:det} \text{(i)};
\item[(ii)] $\eta = d - 4 $ for the model class Model \ref{model:det} \text{(ii.a)};
\item[(iii)] $\eta = 2(d- \beta)$ for the model class Model \ref{model:det} \text{(ii.b)};
\item[(iv)] $\eta = d - 2 $ for the model class Model \ref{model:det} \text{(ii.c)}.
\end{enumerate}
\end{theorem}

\begin{remark}
In view of Theorem \ref{thm:CLT_det}, we observe that we have CLTs for linear statistics with a convergence rate of $O(\log L \cdot L^{-\eta/2})$, where the value of $\eta$ can be obtained from Theorem \ref{thm:var_det} for various model classes of interest. 
\end{remark}

We note that these conditions in Model \ref{model:det} are very general and cover most kernel classes of interest. In Section \ref{s:alpha_det}, we will demonstrate these conditions in some important cases. For instance, (i) above is a quantitative version of the statement that $K(x,x)>0$ for a.e. $x \in \R^d$; whereas (ii.a) entails positivity of the kernel near the diagonal, and the alternative (ii.b) covers the case where the kernel might vanish near the diagonal but does not decay too fast away from it. In particular, this addresses the setting of slowly decaying kernels, such as the Bessel kernel (c.f. Example \ref{ex:Fourier}), which are  known to pose a particularly difficult challenge in stochastic geometry.
For (ii.c), the condition \eqref{eq:3moment} is easily satisfied when the kernel $K$ decays fast enough away from the diagonal; whereas \eqref{eq:rot_inv} holds for any translation and rotation invariant kernel $K$. For more details on variance growth considerations, we refer the reader to Section \ref{s:vargrowth}. 

Taken together, these conditions cover nearly all translation invariant kernel classes of interest for $d>1$; but significantly, they also cover perturbations thereof (e.g. via conjugation by bounded functions), whereas such perturbations may render ineffective  Fourier analytic or other structure-based techniques.
For instance, the intuition behind (ii.b) in Model \ref{model:det} is that there is a polynomial lower bound on the decay of $K(x,y)$ in the separation $\|x-y\|$, but we do not assume it uniformly for all well-separated pairs $x,y$; instead we require that it holds on a positive fraction of the space. This allows for zeros of the kernel $K(x,y)$, which is necessary for many applications, see Example \ref{ex:Fourier}.

We further observe that our approach has limited sensitivity to the ambient dimension $d$, and is particularly effective in dimensions $d>2$, where connections of determinantal processes to random matrices are not available.

Thus, Model \ref{model:det} implies power law lower bounds on $\var[\L(\varphi_L)]$, which translates into power law decay in Theorem \ref{thm:var_det}. This is significant in the context of CLTs for linear statistics of $\a$-determinantal processes, where literature on rates of convergence to Gaussianity is limited (even for the determinantal case $\a=-1$).

The case $\a=0$ in $\a$-determinantal processes corresponds to Poisson point processes, for which CLTs for linear statistics are understood to be simpler in nature due to spatial independence. This case can also be covered under the ambit of our technique; however, we skip the details for reasons of brevity, and note in passing that our necessary estimates will follow directly from the well-known Campbell formula for the Poisson process \cite{Ka}.




\section{A technical tool: a quantitative Marcinkiewicz theorem}
\label{s:marcinkiewicz}

A key ingredient of our CLTs in the present article is a \emph{quantitative version} of the Marcinkiewicz theorem. To elaborate, let $X$ be a real-valued random variable and $\Psi_X(u) := \E[e^{iuX}]$ be the characteristic function of $X$. The classical Marcinkiewicz theorem states that:
\begin{theorem}[Marcinkiewicz, \cite{Mar}]
If $\Psi_X(u) = \exp (P(u))$ for some polynomial $P(x)$, then either $X$ is degenerate (i.e., almost surely a constant) or $X$ is a Gaussian.
\end{theorem}

We remark that the characteristic function $\Psi_X$ being of the form $\exp(f)$ for some entire function $f$ is equivalent to the assertion that $\Psi_X$ is well-defined and has no zeros on the whole of $\C$. Thus, non-vanishing of the characteristic function, coupled with control on its growth rate, can in principle lead to an alternative characterization of the Gaussian distribution. For a brief review on Marcinkiewicz-type theorems and related literature, we refer readers to our Appendix \ref{s:proof-1}.

However, the Marcinkiewicz theorem, in and of itself, is of limited effectiveness as a robust technique for demonstrating Gaussian fluctuations, especially if one is interested in obtaining sharp rates of convergence. Indeed, the assumptions on the well-definedness and the non-vanishing property of the characteristic function on the whole complex plane are often prohibited in applications.

We now introduce a quantitative version of the Marcinkiewicz theorem, which provides upper bounds on the Kolmogorov-Smirnov distance (\textit{abbrv} KS distance) between a (centered) random variable $X$ and the standard Gaussian $N(0,1)$, while only assuming a \emph{bounded} zero-free region for the characteristic function. The relaxation on hypotheses, as remarked above, is crucial in our applications. To elaborate, let us denote the cumulative distribution function (CDF) of a random variable $X$ by $F_X(x)$, and that of a standard Gaussian $N(0,1)$ by $\Phi(x)$. We further denote by $\bar X:= X - \mathbb E[X]$, $\hat X:= (X-\mathbb E[X])/\sqrt{\var[X]}$, and define $\log^+:=\max(\log,0)$.
Our quantitative Marcinkiewicz theorem is stated as follows:

\begin{theorem} \label{t:main-1}
Suppose there is a number $r > 0$ such that $\E[e^{r|X|}]$ is finite and
the characteristic function $u\mapsto \E[e^{iuX}]$ does not vanish on the closed disk $\overline{\D(0,r)}$ of center $0$ and radius $r$.
Then, we have for some universal constant $A>0$
$$\sup_{x\in\R} |F_{\bar X}(x)-\Phi(x)|\leq 2|\sigma-1| + A (1+\log^+\log \max_{|u|=r}|\E[e^{iuX}]|) r^{-1}.$$
\end{theorem}

As a corollary, a quantitative version of the central limit theorem can be deduced from Theorem \ref{t:main-1}. Let $\{X_n\}_{n\ge 1}$ be a sequence of real-valued random variables of variances $\sigma_n^2$ with $\sigma_n>0$.
We denote the associated normalized and centered random variables by
$$\hat X_n:={\bar X_n\over \sigma_n}={X_n-\E[X_n] \over \sigma_n}\cdot$$

\begin{corollary} \label{c:CLT}
Assume that there are positive real numbers $r_n$ such that
$\E[e^{r_n|X_n|}]$ is finite and
the characteristic function $u\mapsto \E[e^{iuX_n}]$ does not vanish on the closed disk $\overline{\D(0,r_n)}$. Assume also that
$$\lim_{n\to\infty}  {1+\log^+ \log \sup_{|u|=r_n} |\E[e^{iu X_n}]| \over r_n\sigma_n} = 0.$$
Then, the sequence $(X_n)$ satisfies the CLT, that is, $\hat X_n$ converges in law to the Gaussian $N(0,1)$ as $n$ tends to infinity. Moreover, we have for some universal constant $A>0$
\begin{equation} \label{eq:CLT_rate}
\sup_{x\in\R} |F_{\hat X_n}(x)-\Phi(x)|\leq {A(1+\log^+ \log \sup_{|u|=r_n} |\E[e^{iu X_n}]|) \over r_n\sigma_n} \cdot
\end{equation}
\end{corollary}

At this point, we compare our rate estimates with the classical Berry-Esseen bounds for sums of random variables. For concreteness, we focus on the setting of sums of independent and identically distributed (i.i.d.) bounded random variables $S_n=\sum_{i=1}^n X_n$.  Berry-Esseen  theorem \cite{Berry,Es-1,Es-2,Dur} gives a convergence rate of $n^{-1/2}$ to the standard normal for $\hat{S}_n$; this rate is optimal in the situation under consideration. We now apply Corollary \ref{c:CLT} to this setting.

Without loss of generality, we assume that the $X_i$-s are centered. We notice that if $X$ is a random variable following the common distribution of $X_i$-s, then the characteristic functions satisfy $\Psi_{S_n}(u)=\l(\Psi_X(u)\r)^n$. Since $\Psi_X(u)$ does not have zeros arbitrarily close to $0$, we may conclude that $\Psi_{S_n}$ does not have zeros in a neighbourhood of radius $\del$ of $0$, where $\del$ does not depend on $n$. This shows that in Corollary \ref{c:CLT}, we may take $r_n=\del$. On the other hand, a simple calculation shows that $\s_n= c \cdot \sqrt{n}$ for some constant $c>0$. Finally, $\E[e^{r_n|S_n|}] \le \Psi_{|X|}(\del)^n$, implying that $\log^+\log \E[e^{r_n|S_n|}] \lesssim \log n$.

Combining these ingredients, we obtain a CLT convergence rate of $\log n \cdot   n^{-1/2}$ for $\hat{S}_n$, which differs from the classical Berry-Esseen rate by only a factor of $\log n$. On the other hand, the Berry-Esseen approach is well-suited for sums of random variables which are preferably independent or weakly dependent; whereas our techniques apply much more generally to strongly correlated settings, and does not assume any algebraic structure (such as sums of smaller ingredients) on the sequence of random variables under consideration.

A key step in the proof of Theorem \ref{t:main-1} is controlling the cumulant sequence. Roughly speaking, the non-vanishing property of the characteristic function over some domain in $\C$ guarantees the existence and the holomorphicity of the cumulant generating function on that domain, which allows us to analyze the associated cumulant sequence. This has been accomplished by Michelen and Sahasrabudhe using a complex analytic-probabilistic argument in their seminal works \cite{MS-1,MS-2}. Primarily, these works investigate CLTs for non-negative integer-valued random variables under conditions on zero-free regions for their generating polynomials, in relation to the earlier works \cite{LPRS}, \cite{GLP} and a related variance growth question due to Pemantle. However, in \cite{MS-2} Section 12, they also provide a CLT for general real-valued random variables, which is an extension of their approach. To demonstrate their result, let $X$ be a real-valued random variable with standard deviation $\sigma>0$. The \emph{probability generating function} of $X$ is defined as $f_X(z):= \E[ z^X]$ for $z\in \C\setminus \R_{\le 0}$, where $z^x:= \exp (x \log z)$ for $x\in \R, z\in \C\setminus \R_{\le 0}$, and $\log$ denotes the principal branch of the complex logarithm.
\begin{theorem} [Theorem 12.2 in \cite{MS-2}] \label{t:MS-2}
Let $X$ be a real-valued random variable. Assume that $f_X(\rho)<\infty$ for all $\rho \ge 0$ and there exists $\delta>0$ such that $f_X(z)$ has no zeros in the sector
\[S(\delta):= \{z \in \C\setminus\{0\}: -\delta \le \arg z \le \delta\}, \]
where $\arg$ denotes the principal argument.
If $u_X(z):= \log |f_X(z)|$ satisfies
\[ \lim_{|z|\rightarrow \infty, z\in S(\delta)} \frac{u_X(|z|)}{|z|^\kappa} = 0 \quad \text{and}\quad \lim_{|z| \rightarrow \infty, z\in S(\delta)} \frac{|u_X(1/z)|}{|z|^\kappa} =0\]
for some $\kappa>0$, then
\[ \sup_{x\in \R}|F_{\hat X}(x)- \Phi(x) | = O \Big ( \frac{\max\{\delta^{-1},\kappa\}}{\sigma} \Big ) \cdot \]
\end{theorem}

In the coordinate $u$ given by $z=e^u$, the assumptions in Theorem \ref{t:MS-2} translate to that $\mathbb E[e^{uX}]< \infty$ for all $u\in \R$ (which is a strong moment condition on $X$), and that $\mathbb E[e^{uX}]$ is zero-free on the infinite strip $\{ u \in \C: |\Im (u)| \le \delta\}$.
In a similar flavour to this result, Theorem 2 in \cite{EF} also obtains a CLT for real-valued random variables under the assumptions that the characteristic  function is entire, zero-free on an infinite vertical strip and satisfies global growth conditions. 

A significant point to note about Theorem \ref{t:main-1} and Corollary \ref{c:CLT} in this paper is that they provide effective versions of the quantitative CLT that are applicable to wide classes of random variables of interest in statistical mechanics, including in situations that are not covered by existing results of similar flavour. In particular, we only need to assume that the characteristic function is finite on a (small, possibly shrinking) disk around the origin,   zero-free on this  disk, and a growth estimate that needs to hold only on this disk. The limited and local nature of these assumptions are crucial for applications, demonstrated in Remarks \ref{rem:spin} and \ref{rem:det}. This includes, in particular, the applications to continuous and multi-component spin systems and strongly correlated random point fields considered in the present paper.

While Theorem \ref{t:main-1} can be obtained by following a similar approach as \cite{MS-2}, for completeness we provide proofs of Theorem \ref{t:main-1} and Corollary \ref{c:CLT} in  the appendix Section \ref{s:proof-1}. It may be noted that our proofs use a simpler and more succinct argument, purely based on classical complex analytic ideas.

\section{CLTs in general spin systems:  generalised Ising, XY and Heisenberg models}
\label{s:spin_systems}
\subsection{CLTs for free boundary condition spin systems}
We recall the definition of free boundary condition spin systems from Section \ref{s:appl_spin}; in particular, the Hamiltonian $H_\L$ \eqref{eq:Hamiltonian} and the spin measure $\P_\L(\bs_\L)$ \eqref{eq:measure}. Recall that the spin $\bs_x \in \R^N$ for each $x \in \L$. We will consider one-component $(N=1)$ models with general real valued spins (essentially, generalised versions of the classical Ising model), and multi-component spin systems with $N\geq 2$ components.

We would be interested in the \textit{partition function} for the spin system, which is given by the integral
\begin{equation} \label{eq:partition_function}
 \mathcal{Z}_{\beta,\Lambda} (\{ \bm h_x\}_{x\in \Lambda} ):=\int_{(\R^N)^{\otimes |\Lambda|}}\exp(- \beta H_{\Lambda}(\bm \sigma_{\Lambda})) \prod_{x \in \Lambda} \d \mu_0(\bm \sigma_x).
\end{equation}

\subsection{Ferromagnetism and its generalisations}

When $N=1$, the model is called \textit{ferromagnetic} if $J_{xy}=J_{yx} \ge 0$ for all $x,y$. For ferromagnetic models with homogeneous nearest neighbour interaction  $J$ and uniform magnetic field $h$, it is known that, at positive temperature $\beta^{-1}$, the total spin $\sigma$ exhibits a Gaussian central limit theorem as the $d$-dimensional cubic domain $\Lambda \uparrow \mathbb Z^d$ (see, e.g.,  \cite{Ellis_book}).

For multi-component spins (i.e., $N \ge 2$), the generalisation of the ferromagnetic condition is of considerable interest, but not straightforward. It is considered in the literature that for the XY model, the appropriate generalisation is $J_{xy}^1 \ge |J_{xy}^2|$, whereas for the classical Heisenberg model, the analogous condition is $J_{xy}^1 \geq \max\{|J_{xy}^2|, |J_{xy}^3|\}$; recall the spin Models \ref{model:spin}.

For $N \ge 3$, to the best of our knowledge, entirely satisfactory  generalisations are not known,  even for special choices of the background spin measure $\mu_0$. A natural extension of ferromagnetism for the XY and classical Heisenberg models to higher component spins would be to posit that $J_{xy}^1 \geq \max \{|J_{xy}^i|: 2 \leq i \leq N\}$. Unfortunately, crucial theorems on such spin models (such as the Lee-Yang theorem for the zeros of the partition function, see Section \ref{s:Lee-Yang}) are known to hold under a much more restrictive condition, namely, $J_{xy}^1 \geq \sum_{i=2}^N |J_{xy}^i|$.


\subsection{Lee-Yang theorem} \label{s:Lee-Yang}

We will now discuss the celebrated Lee-Yang theorem for  spin systems, which will be a key tool in verifying the zero-free condition with regard to CLT for the total spin. To this end, we first recall the definition of the subset $\Omega^N_\del \subset \C^N$ with $\del \in \R$ in \eqref{eq:goodvector},
and we denote by $\Int(B)$ the interior of a subset $B\subset \C^N$.
The Lee-Yang theorem is stated as follows:

\begin{theorem}[Lee-Yang theorem] \label{thm:Lee-Yang}
Let the random spin configuration \eqref{eq:measure} with Hamiltonian \eqref{eq:Hamiltonian}  satisfy any of the conditions  (a) -- (d) in Model \ref{model:spin}. Then for $\beta>0$, the partition function $\mathcal Z_{\beta,\Lambda}(\{\bm h_x\}_{x\in\Lambda})$  in \eqref{eq:partition_function} does not vanish whenever $\bm{h}_x \in \Int(\Omega^N_0)$ for every $x \in \Lambda$.
\end{theorem}

For detailed discussion on Lee-Yang theorem, we refer the reader to the articles \cite{FR,LS,New}, and the references therein.

\subsection{The no-zeros condition for the total spin}

Herein we establish that the desired no-zeros condition for the characteristic function of the total spin $\bm S_\L$ for ferromagnetic systems at positive temperature $\beta^{-1}$. 

\begin{lemma} \label{lem:no_zeros}
Let the random spin configuration \eqref{eq:measure} with Hamiltonian \eqref{eq:Hamiltonian}  satisfy any of the conditions  (a) -- (d) in Model \ref{model:spin}. 
Let $\bm \alpha$ be any unit vector in $\R^N$.
Then there exists $r>0$ (not depending on $\L$ and $\bm\alpha$) such that $\E[\exp(u \langle \bm S_\L, \bm\alpha \rangle ) ]$ does not vanish whenever $u \in \overline{\mathbb D(0,r)}$.
\end{lemma}

\begin{proof}[Proof of Lemma \ref{lem:no_zeros}]
The three integrals below are with respect to the measure $\prod_{x\in\Lambda} \d \mu_0(\bm\sigma_x)$
\begin{align*}
& \E[\exp(u \langle \bm S_\L, \bm\alpha \rangle)] \\
=& {\int_{(\R^N)^{\otimes |\Lambda|}} \exp \Big (u\sum_{x \in \Lambda} \langle \bs_x,\bm \alpha\rangle \Big) 
\exp (-\beta H_{\Lambda}(\bm \sigma_{\Lambda}) )  \over \mathcal{Z}_{\beta,\Lambda}(\{\bm h_x\}_{x\in\Lambda}) } \\
=& {\int_{(\R^N)^{\otimes |\Lambda|}} \exp \Big (u\sum_{x \in \Lambda} \langle \bs_x,\bm \alpha\rangle \Big) 
\exp\Big (\beta\sum_{(x,y) \subset \Lambda} \sum_{i=1}^N J_{xy}^i\sigma_x^i\sigma_y^i + \beta \sum_{x \in \Lambda} \langle \bs_x, \bm h_x \rangle \Big )  \over \mathcal{Z}_{\beta,\Lambda}(\{\bm h_x\}_{x\in\Lambda}) } \\
=& {\int_{(\R^N)^{\otimes |\Lambda|}} \exp\Big (\beta\sum_{(x,y) \subset \Lambda} \sum_{i=1}^N J_{xy}^i\sigma_x^i\sigma_y^i 
+ \beta \sum_{x \in \Lambda} \langle \bs_x, \bm h_x + \beta^{-1} u \bm \alpha \rangle\Big )  
\over \mathcal{Z}_{\beta,\Lambda}(\{\bm h_x\}_{x\in\Lambda}) } \\
=& \frac{\mathcal{Z}_{\beta,\Lambda}(\{\bm h_x + \beta^{-1} u \bm \alpha\}_{x\in\Lambda})}{\mathcal{Z}_{\beta,\Lambda}(\{\bm h_x\}_{x\in\Lambda})} \cdot \numberthis \label{eq:ratio_partition}
\end{align*}

By \eqref{eq:ratio_partition}, $\E[\exp(u \langle \bm S_\L, \bm\alpha \rangle ) ] =0$ if and only if $\mathcal{Z}_{\beta,\Lambda}(\{\bm h_x + \beta^{-1} u \bm \alpha\}_{x\in\Lambda})=0$. 
Denote by $\bm v_x:= \bm h_x + \beta^{-1} u \bm\alpha$ for each site $x\in \L$, we then have
\begin{eqnarray*}
\Re  v_x^1- \sum_{j=2}^N |v_x^j| &=& \Re (h_x^1 + \beta^{-1}u\alpha_1) - \sum_{j=2}^N |h_x^j + \beta^{-1}u\alpha_j| \\
&\geq& \Re h^1_x - \beta^{-1}|u||\alpha_1| -  \sum_{j=2}^N |h_x^j| - \sum_{j=2}^N \beta^{-1}|u||\alpha_j| \\
&=& \Big (\Re h^1_x - \sum_{j=2}^N |h^j_x| \Big ) - \beta^{-1}|u| \sum_{j=1}^N |\alpha_j| \\
&\ge& \del - \beta^{-1} |u| \sqrt{N}.
\end{eqnarray*}
Hence, there exists $r>0$ (depending only on $\del, \beta$ and $N$) such that 
$  \Re v^1_x - \sum_{j=2}^N |v^j_x| >0$  whenever  $|u|\leq r$.
This implies $\bm v_x \in \Int(\Omega^N_0)$ for every $x\in \L$. By the Lee-Yang theorem \ref{thm:Lee-Yang}, $\E[\exp(u \langle \bm S_\L, \bm\alpha \rangle ) ]$ is zero-free on $\overline{\mathbb D(0,r)}$.
\end{proof}

\begin{remark} \label{rem:spin}
In the setting of a uniform magnetic field $\bm{h}$ and $\bm \alpha = \bm e_1:=(1,0,\ldots,0)$ (for simplicity), we observe in the proof of Lemma \ref{lem:no_zeros} that for $\Psi_{\langle \bm S_\L, \bm \alpha \rangle} (u) := \mathbb E[ e^{iu \langle \bm S_\L, \bm \alpha \rangle}] \neq 0$,  we need $\bm{h} + \beta^{-1}iu \bm{e_1} \in \Int(\Omega^N_0)$, or equivalently, $\Re{h^1}  - \b^{-1} \Im{u} > \sum_{i=2}^N |h^i|$.  Clearly, this requires us to be able to choose $|\Im(u)|$ to be small. In particular, $u$ cannot be allowed to vary over an infinite vertical strip in $\C$.
\end{remark}


\subsection{Variance growth}
In this section, we demonstrate variance growth conditions for $\langle \bm S_\L, \bm \alpha \rangle$, which ensure asymptotics normality of $\langle \bm S_\L, \bm \alpha \rangle$.

\subsubsection{Upper bound on \texorpdfstring{$\langle \bm S_\L, \bm \alpha \rangle$}{<S_L, alpha>}} \label{s:ubound_spin}

For a compactly supported spin distribution $\mu_0$, there exists $M>0$ (depending only on $\mu_0$) such that $\|\bm{\sigma}_x\| \le M$ a.s. for every $x\in \L$. In view of this, we easily bound
$$|\langle \bm S_\L, \bm \alpha \rangle|
 = \Big |\sum_{x \in \Lambda} \langle \bs_x, \bm \alpha \rangle \Big | 
 \le \sum_{x\in \L} \|\bm \sigma_x\|  \le |\Lambda| \cdot M , \quad \text{ almost surely.}$$
 
The hypothesis of a compactly supported spin distribution $\mu_0$ is valid for most natural models, including the Ising model, the XY model and the classical Heisenberg model (where the spin measure $\mu_0$ is supported on $\mathbb{S}^0,\mathbb{S}^1$, and $\mathbb{S}^2$ respectively).

\subsubsection{Lower bound on the variance of \texorpdfstring{$\langle \bm S_\L, \bm \alpha \rangle$}{<S_L, alpha>}} \label{s:lbound_spin}
Now we discuss how to obtain lower bounds on the variance growth rate of $\langle \bm S_\L, \bm \alpha \rangle$ that scale as some power $|\Lambda|^\epsilon$ of the total system size (i.e. volume) $|\Lambda|$. To this end, we have the following lemma.

\begin{lemma} \label{lem:local_spin_var}
Consider a free boundary condition spin system satisfying any of the conditions  (a) -- (d) in Model \ref{model:spin}.
For each site $x\in \Lambda$, denote by $\partial x$ the set of all vertices in $\Lambda$ that are connected by an edge to $x$. Let $\bm \alpha\in \mathbb S^{N-1}$ be a fixed unit vector.
Then
\begin{enumerate}
\item[(i)] There exists a constant $c>0$ (not depending on $\Lambda$) such that
$$ \var[\langle \bm\sigma_x,\bm\alpha \rangle | \bs_y, y\in \partial x] \geq c, \quad \forall x \in \Lambda.$$
\item[(ii)] $\var[\langle \bm S_\L, \bm \alpha \rangle] \gtrsim |\Lambda|$.
\end{enumerate}
\end{lemma}

\begin{proof}

\noindent
(i) Given the spins on $\partial x$, the conditional distribution for the spin at the site $x$ has the following structure
\begin{equation} \label{eq:spin-conditional-dist}
\mathbb P_{\Lambda} ( \bm \sigma_x| \bm \sigma_{y}, y\in \partial x) \propto \exp \Big ( \beta\sum_{i=1}^N \sigma^i_x \Big (h^i_x + \sum_{y\in \partial x} J^i_{xy} \sigma^i_y \Big) \Big ) \d \mu_0(\bm \sigma_x) = \exp(\langle \bs_x, \bm \gamma \rangle) \d \mu_0(\bs_x),
\end{equation}
where $\bm \gamma = (\gamma_1,\ldots,\gamma_N) \in \R^N$, $\gamma_i := \beta (h^i_x + \sum_{y\in \partial x} J^i_{xy} \sigma^i_y)$,  for $1\le i \le N$.

\medskip

We recall that the coupling constants and the external magnetic fields are assumed to be uniformly bounded, i.e., there exist constants $H, J > 0 $ (not depending on $\L$) such that $|h^i_x| \leq H$ and $ |J^i_{xy}| \leq J$ for every $x,y \in \L$. 
Thus, there exists $M>0$ (not depending on $\L$) such that $\|\bm \gamma\| \le M$. 

\medskip

By \eqref{eq:spin-conditional-dist}, one has
$$\var[\langle \bm\sigma_x,\bm\alpha\rangle |\bm \sigma_{y}, y\in \partial x] = {f(\bm \gamma) \over h(\bm \gamma)} - \Big ({g(\bm \gamma) \over h(\bm \gamma)} \Big )^2,$$
where the functions $f(\bm u),g(\bm u)$ and $h(\bm u)$ are defined for $\bm u=(u_1,\ldots,u_N) \in \R^N$ as follows
\begin{align*}
f(\bm u) &:=  \int_{\mathbb R^N} \langle \bm\sigma_x,\bm\alpha\rangle^2 \exp (\langle \bs_x, \bm u\rangle ) \d \mu_0(\bm \sigma_x),  \\
g(\bm u) &:= \int_{\mathbb R^N} \langle \bm\sigma_x,\bm\alpha\rangle \exp (\langle \bs_x, \bm u\rangle ) \d \mu_0(\bm \sigma_x),  \\
h(\bm u) &:= \int_{\mathbb R^N}  \exp (\langle \bs_x, \bm u\rangle ) \d \mu_0(\bm \sigma_x).
\end{align*}
Since $\mu_0$ is compactly supported, $f(\bm u) , g(\bm u)$ and $h(\bm u)$ are finite and continuous on $\R^N$.
Let
$F(\bm u):= (f(\bm u) h(\bm u) - g(\bm u)^2)/ h(\bm u)^2.$ It follows that $F(\bm u)$ is also continuous. Moreover, by the Cauchy-Schwarz inequality, we have
$$f(\bm u)h(\bm u) \geq \Big (\int_{\R^N} |\langle \bs_x, \bm\alpha \rangle | \exp(\langle \bs_x, \bm u\rangle ) d\mu_0(\bs_x) \Big )^2 \geq g(\bm u)^2.$$
Moreover, for $f(\bm u) h (\bm u) - g(\bm u)^2 = 0$, we must have
$\langle \bm\sigma_x,\bm\alpha\rangle =$  constant $\mu_0$-a.s., which is not possible in our model. Hence, $F(\bm u)>0$ for all $\bm u \in \mathbb R^N$.
Let
$c:= \min_{ \|\bm u\| \leq M} F(\bm u)>0,$
then $c$ does not depend on $\L$, and we have
\[\var[\langle \bm\sigma_x,\bm\alpha\rangle| \bm \sigma_{y}, y\in \partial x] = F(\bm \gamma) \geq c>0.\]

\medskip
\noindent
(ii) We decompose $\L = \L_E \cup \L_O$, where
$$\Lambda_E:= \{x=(x_1,\ldots,x_d)\in \Lambda: x_1+\ldots+x_d \text{ is even}\}$$ 
$$\Lambda_O:= \{x=(x_1,\ldots,x_d)\in \Lambda: x_1+\ldots+x_d \text{ is odd}\}.$$
For any vertex $x\in \Lambda_E$, we have $\partial x \subset \Lambda_O$. Thus, the conditional distribution for the spins on $\L_E$ given the spins on $\L_O$ has the following structure
\begin{eqnarray*}
\mathbb P_\Lambda(\bs_x, x\in \Lambda_E|\bs_y, y\in \Lambda_O) &\propto& \exp \Big (\beta \Big (\sum_{x\in \Lambda_E} \sum_{y\in\partial x} \sum_{i=1}^N J^i_{xy} \sigma_x^i \sigma_y^i + \sum_{x\in \Lambda_E} \sum_{i=1}^N h^i_x \sigma^i_x\Big ) \Big ) \prod_{x\in \Lambda_E} \d \mu_0(\bs_x) \\
&=& \prod_{x\in \Lambda_E} \exp \Big (\beta \Big ( \sum_{y\in\partial x} \sum_{i=1}^N J^i_{xy} \sigma_x^i \sigma_y^i +  \sum_{i=1}^N h^i_x \sigma^i_x\Big ) \Big ) \d \mu_0(\bs_x).
\end{eqnarray*}
This implies that, given the spins $\{\bs_y\}_{ y\in \Lambda_O}$, the spins $\{\bs_x\}_{ x\in \Lambda_E}$ are independent. 

By the law of the total variance, one has
\begin{eqnarray*}
\Var[ \langle \bm S_\L, \bm \alpha \rangle ] &\geq& \mathbb E[\var[\langle \bm S_\L, \bm \alpha \rangle | \bs_y, y\in \Lambda_O]] \\
&=& \mathbb E\Big [\var\Big [ \sum_{x\in \Lambda} \langle \bm\sigma_x,\bm\alpha\rangle | \bs_y, y\in \Lambda_O \Big ]\Big ] \\
&=&  \mathbb E\Big [\var\Big [ \sum_{x\in \Lambda_E} \langle \bm\sigma_x,\bm\alpha\rangle | \bs_y, y\in \Lambda_O \Big ]\Big ] \\
&=&  \sum_{x\in \Lambda_E} \mathbb E [\var[ \langle \bm\sigma_x,\bm\alpha\rangle | \bs_y, y\in \Lambda_O]] \\
&\geq& c \cdot |\Lambda_E| \approx {c\over 2} \cdot |\Lambda|.
\end{eqnarray*}
\end{proof}

\subsection{Proof of Theorem \ref{thm:CLT_spin}}
\begin{proof}[Proof of Theorem \ref{thm:CLT_spin}]
Herein, we put together the various ingredients developed in the earlier subsections, and complete the proof of Theorem \ref{thm:CLT_spin}. 
Given an unit vector $\bm \alpha$, we want to show a CLT for $\langle \bm S_\L, \bm \alpha \rangle$ as $|\L|\uparrow \infty$.
In view of Corollary \ref{c:CLT}, for every $\L$ large enough, we need to specify $r_\L$ such that the characteristic function   $\Psi_{\langle \bm S_\L, \bm \alpha \rangle}(u) := \mathbb E[\exp(iu \langle \bm S_\L, \bm \alpha \rangle)]$ does not vanish on $\ol{\D(0,r_\L)}$; furthermore, we need to verify the growth condition in that corollary.

Thanks to Lemma \ref{lem:no_zeros}, there exists $r>0$, not depending on $\Lambda$, such that $\Psi_{\langle \bm S_\L, \bm \alpha \rangle}(u)$ is holomorphic and zero-free on $\overline{\mathbb D(0,r)}$, for every $\Lambda$.  
Thus, it remains to verify the growth condition 
$$\lim_{|\Lambda| \uparrow \infty} {\log^+ \log |\mathbb E[e^{r|\langle \bm S_\L, \bm \alpha \rangle|}]| \over \sqrt{\var[\langle \bm S_\L, \bm \alpha \rangle]}} = 0.$$

 To this end, from Section \ref{s:ubound_spin} we have $|\langle \bm S_\L, \bm \alpha \rangle| \lesssim |\Lambda|$, which implies
 $$\log^+ \log |\mathbb E[e^{r|\langle \bm S_\L, \bm \alpha \rangle|}]| \lesssim \log |\Lambda|.$$
From Lemma \ref{lem:local_spin_var}, we have $\sqrt{\var[\langle \bm S_\L, \bm \alpha \rangle]} \gtrsim |\Lambda|^{1/2}$.
By applying these bounds in the CLT convergence rate \eqref{eq:CLT_rate}, we conclude that the Kolmogorov-Smirnov distance between the normalized random variable $(\langle \bm S_\L, \bm \alpha \rangle - \mathbb E[\langle \bm S_\L, \bm \alpha \rangle]) / \sqrt{\langle \bm S_\L, \bm \alpha \rangle}$ and a standard Gaussian decays as $O({\log |\L|} \cdot |\L|^{-1/2})$.
\end{proof}

\subsection{CLTs for spin systems with boundary conditions}

We recall the definition of spin systems with boundary conditions from Section \ref{s:appl_spin}; in particular, let $\bs_\L$ be a spin configuration with the boundary condition $\mathcal V$, with the Hamiltonian $H_\L^{\mathcal V}$ \eqref{eq:cond-Hamiltonian} and the spin measure $\P_\L^{\mathcal V}(\bs_\L)$ \eqref{eq:cond-measure}. With Theorem \ref{thm:CLT_spin} in hand, now we are able to prove Theorem \ref{thm:CLT_spinb}. 
\begin{proof}[Proof of Theorem \ref{thm:CLT_spinb}]
Observe that
\begin{eqnarray*}
H_\L^{\mathcal V}(\bs_\L) &=& - \sum_{(x,y) \subset \L'} \sum_{i=1}^N J^i_{xy} \s_x^i\s_y^i - \sum_{x\in \L} \sum_{i=1}^N h^i_x \s^i_x \\
 &=&- \Big (\sum_{(x,y) \subset \L} \sum_{i=1}^N J^i_{xy} \s_x^i\s_y^i - \sum_{x\in \L} \sum_{i=1}^N h^i_x \s^i_x \Big ) - \sum_{x\in\L, y \in \partial \L, (x,y)\subset \L'} \sum_{i=1}^N J^i_{xy} \s_x^i v_y^i. 
\end{eqnarray*}
For each $x\in \L$ and $i\in \{1,\ldots,N\}$, let
$$\tilde h_x^i := h_x^i + \sum_{y\in \partial \L, (x,y)\subset \L'} J^i_{xy} v^i_y,$$
with the convention that the sum is zero whenever the index set is empty. Then, we can identify $\bs_\L$ with a free boundary condition spin system $\tilde \bs_\L$ defined on the same graph $\L$, with the same coupling constants $J_{xy}^i$, and with new external magnetic fields $\tilde {\bm h}_x:=(\tilde h^1_x,\ldots,\tilde h^N_x)$ at each site $x\in \L$.

Now observe that for every $x\in \L$
\begin{eqnarray*}
\Re \tilde h^1_x - \sum_{i=2}^N |\tilde h^i_x| &\ge& \Big (\Re h^1_x - \sum_{i=2}^N |h^i_x| \Big ) + \sum_{y\in \partial \L, (x,y)\subset \L'} \Big (J^1_{xy} v^1_y - \sum_{i=2}^N |J^i_{xy}|| v^i_y| \Big ) \\
&\ge&  \Big (\Re h^1_x - \sum_{i=2}^N |h^i_x| \Big ) + \sum_{y\in \partial \L, (x,y)\subset \L'} J^1_{xy} \Big ( v^1_y - \sum_{i=2}^N |v^i_y| \Big )\\
&\ge&  \Big (\Re h^1_x - \sum_{i=2}^N |h^i_x| \Big ).
\end{eqnarray*}
Applying Theorem \ref{thm:CLT_spin} to this new spin system gives our desired result.
\end{proof}

\section{Fluctuation theory for linear statistics of $\alpha$-determinantal processes}
\label{s:alpha_det}

In this section, we will undertake a detailed study of $\a$-determinantal processes in the context of our approach to CLTs based on zeros of characteristic functions. We will work in the setting of Section \ref{s:appl_det} in general, and Model \ref{model:det} in particular.

\subsection{The characteristic function of linear statistics} \label{s:analytic_continuation}
We begin with an expression for the Laplace transform of $\Lambda(\varphi_L)$ in terms of Fredholm determinants that is valid for bounded and compactly supported test functions $\ph$ and complex arguments $u$ that are small enough. This is available in the literature \cite{ST,Bac}.
%

To this end, we set up some notations. Let $L^2(\mu)$ be the space of $\C$-valued, square-integrable functions with respect to $\mu$ on $\R^d$. For a kernel $K: \R^d \times \R^d \rightarrow \C$, the associated integrable operator is defined by
\[ f(x) \mapsto \int_{\R^d} K(x,y) f(y) \d \mu(y), \quad f \in L^2(\mu). \]
By abuse of notation, we denote that operator by the same symbol $K$. For a bounded function $g:\R^d \rightarrow \C$, we define the multiplication operator induced by $g$, denoted by $M_g$, by
\[ M_g: L^2(\mu) \rightarrow L^2(\mu)\quad,\quad f(x) \mapsto g(x)f(x).\]
Now we can state:

\begin{proposition} \label{prop:Laplace_det}
Let $X$ be the $\a$-determinantal process as in Theorem \ref{thm:CLT_det}, and 
let $\ph$ be a bounded, compactly supported function on $\R^d$ with $\supp(\ph) =: B$. For each $u\in \C$, we define a function 
\[g_u(x):=1-\exp(-u\varphi(x)),\]
and let $M_{g_u}$ be the multiplication operator induced by the function $g_u$. Let $K_B$ be the operator compression of $K$ given by 
\[K_B(\cdot,\cdot)=\ind_B(\cdot) K(\cdot, \cdot) \ind_B(\cdot).\]
Then there exists $r>0$ (depending only on $\|\varphi\|_\infty$) such that for every $u \in \overline{\mathbb D(0,r)}$
\begin{equation} \label{eq:Fredholm-2}
\E[\exp(-u\Lambda(\varphi))] = \Det[I + \alpha M_{g_u} K ]^{- \frac{1}{\alpha}} = \Det[I + \alpha M_{g_u} K_B ]^{- \frac{1}{\alpha}},
\end{equation}
where $\Det$ is the Fredholm determinant and $I$ is the identity operator on $L^2(\mu)$.  Further, the Laplace transform is holomorphic in $u$ on $\D(0,r)$.
\end{proposition}

\begin{proof}[Proof of Proposition \ref{prop:Laplace_det}]

By Theorem 1.5 in \cite{ST} (see Theorem \ref{thm:ST} below), for $\|u \varphi \|_\infty = |u| \|\varphi\|_\infty$ sufficiently small, we have
$$\E[\exp(-u\Lambda(\varphi))] = \Det [I + \alpha M_{g_u} K ]^{-1/\alpha}.$$
 By Theorem 2.4 in \cite{ST} (for a concrete statement, see Theorem \ref{thm:FD_ST} below), we may deduce that as soon as the operator $\a M_{g_u} K$ satisfies $\|\a M_{g_u} K\|_{\rm op}<1$, we may write
\begin{align*}
\Det[I + \a M_{g_u} K]^{-1/\a}   &=  \sum_{k=0}^{\infty} \frac{1}{k!} \Big (\int^{\otimes k}   \Det_\a[\l(-g_u(x_i) K(x_i,x_j)\r)_{i,j=1}^k]   \d\mu(x_1) \ldots \d\mu(x_k) \Big ) \\
&=  \sum_{k=0}^{\infty} \frac{1}{k!} \Big (\int^{\otimes k}   \Det_\a[\l(-g_u(x_i) K_B(x_i,x_j)\r)_{i,j=1}^k]  \d\mu(x_1) \ldots \d\mu(x_k) \Big ) \\
&= \Det[I+\alpha M_{g_u}K_B]^{-1/\alpha},
\end{align*}
here we have used the fact that the function $g_u(\cdot)$ is supported on $B=\supp(\ph)$, and hence the term $\Det_\alpha[(-g_u(x_i)K(x_i,x_j))_{i,j}^k]$ is non-zero only when $x_i \in B, \forall i$.

For $|u|$ small enough (depending on $\|\varphi\|_\infty$), say $|u|\leq r_0$, one may have 
$$\|M_{g_u}\|_\op = \|1-\exp(-u\varphi(\cdot))\|_\infty \le {1\over 2} |\a|^{-1} \|K\|_{\rm op}^{-1}.$$
For such $u$,
$ \| \alpha M_{g_u} K \|_{\rm op} \leq |\alpha| \|M_{g_u}\|_{\op} \|K\|_{\rm op} < {1\over 2}.$
Consequently, 
$$\E[\exp(-u\Lambda(\varphi))] = \Det [I + \alpha M_{g_u} K ]^{-1/\alpha} =  \Det [I + \alpha M_{g_u} K_B ]^{-1/\alpha} .$$
Note that the operator $K_B$ is of trace class since $K$ is of locally trace class and $B$ is compact. Since $M_{g_u}$ is a bounded operator, we deduce that the operator $ \alpha M_{g_u} K_B$ is of trace class, which implies that $\Det[I + \alpha M_{g_u} K_B ]$ is finite. On the other hand, the only way that $\Det [I + \alpha M_{g_u} K_B]$ can be $0$ is for one of its eigenvalues to be equal to $0$, which cannot happen because all the eigenvalues of $\alpha M_{g_u} K_B$ are of modulus bounded away from $1$. This implies that for $|u|\leq r_0$, the quantity $\Det[I + \a M_{g_u} K_B]^{-1/\a}$  is finite and has no singularities.

For the holomorphicity of the Laplace transform in a neighbourhood of $0$, let $|\varphi|$ be the function defined by $|\varphi|(x) := |\varphi(x)|$. Applying the arguments above for $|\varphi|$, we deduce that there exists $r_1 >0$ depending only on $\|\varphi\|_\infty$ such that
$$\E[\exp(-u\Lambda(|\varphi|))] = \Det [I + \alpha M_{f_u} K_B]^{-1/\alpha} \quad \text{ whenever } |u| \leq r_1$$
where $f_u(x) := 1 - \exp(-u|\varphi(x)|)$. In particular, this implies 
$\E[\exp(r_1\L(|\varphi|))]$ is finite.
Hence,
$ \E[\exp(r_1|\L(\varphi)|)]  \leq \E[\exp(r_1\L(|\varphi|))] < \infty.$
Thus $ \E[\exp(u\L(\varphi))]$ is a holomorphic function in $u$ on $\D(0,r_1)$. Setting $r:= \min(r_0,r_1)$ completes the proof.
\end{proof}

We add here a statement of Theorem 1.5 and Theorem 2.4 in \cite{ST} that was used in the above proof.

\begin{theorem}[Theorem 1.5 in \cite{ST}] \label{thm:ST}
Let the kernel $K$ satisfy condition (i) and (ii) in Theorem \ref{thm:CLT_det}. Then we have
$$\E[\exp(-\L(f))] = \Det [I+ \a M_f K]^{-1/\a}$$
for any complex-valued, bounded, compactly supported test function $f$ with $\|f\|_\infty$  is sufficiently small, here $M_f$ denotes the multiplication operator by the function $1-e^{-f}$, acting on $L^2(\mu)$.
\end{theorem}

\begin{theorem}[Theorem 2.4 in \cite{ST}] \label{thm:FD_ST}
Let $J$ be a trace class integral operator on $L^2(\mu)$. If $\| \a J \|_\op < 1$, we have 
\[  (\Det[I - \a J])^{-1/\a}  = \sum_{n=0}^\infty \frac{1}{n!} \int^{\otimes n}  \Det_\a[\l( J(x_i,x_j)\r)_{i,j=1}^n]  \ \d\mu(x_1) \ldots \d\mu(x_n). \]
\end{theorem}

\begin{remark} \label{rem:det}
Let $M_{g_{-iu}}$ denote the multiplication operator by the function $1-\exp(iu \varphi(\cdot))$ acting on $L^2(\mu)$, we have 
$$\Psi_{\L(\varphi)}(u) = \E[\exp(iu \L(\varphi))] = \Det [I + \a M_{g_{-iu}}K]^{-1/\a}$$ 
provided that $\| \a M_{g_{-iu}} K \|_\op < 1$. We observe that the proof of Proposition \ref{prop:Laplace_det} entails $|u|$ small implying $\|1-e^{iu \ph(\cdot)}\|_\infty$ small, in turn implying $\|\a M_{g_{-iu}} K \|_{\rm op}<1$; this would imply non-vanishing of the Fredholm determinant. However, if $|\Im(u)|$ was large, $\|1-e^{iu \ph(\cdot)}\|_\infty$ would be large, and the above argument would break down. In particular, if $u$ is allowed to vary in an infinite vertical strip in $\C$, then it might be possible to choose $u$ such that $-1$ is an eigenvalue of $ \a M_{g_{-iu}} K $, and hence such $u$ would be a zero of the Fredholm determinant. This would lead to the vanishing of the characteristic function $\Psi_{\Lambda(\varphi)}$ for $\a<0$, and to its blow up for $\a>0$.
\end{remark}

Proposition \ref{prop:Laplace_det} implies the zero-freeness and holomorphicity of $\E[\exp(u\L(\varphi))]$ for any bounded, compactly supported function $\varphi$ on a small disc $\D(0,r)$, where $r$ depends only on $\|\varphi\|_\infty$. Combining this with the observation that $\|\varphi_L\|_\infty = \|\varphi(\cdot/L)\|_\infty = \|\varphi\|_\infty$ for every $L>0$, we deduce the following result:

\begin{prop} \label{prop:0-condition}
Let $X$ be an $\alpha$-determinantal process as in Theorem \ref{thm:CLT_det}, and $\varphi: \R^d \rightarrow \R$ be a bounded compactly supported function. Then there exists $r>0$ (depending only on $\|\varphi\|_\infty$) such that for every $L>0$, the characteristic function $\Psi_{\Lambda(\varphi_L)}(u)=\E[\exp(iu \Lambda(\varphi_L))]$ is holomorphic and nonzero whenever $u \in \D(0, r)$.
\end{prop}

\subsection{Growth rate of the moment-generating function of $\Lambda(\varphi_L)$}

In this section, we provide a result about the growth rate of the moment-generating function of $\Lambda(\varphi_L)$, which is necessary in our approach.

\begin{proposition} \label{prop:mgf-growth}
Let $X$ be an $\alpha$-determinantal process as in Theorem \ref{thm:CLT_det}. Let $\varphi :\R^d \rightarrow \R$ be a bounded compactly supported test function. There exist constants $r,c> 0$ such that for $L >0$ we have
$$ \E[\exp(r|\Lambda(\varphi_L)|)] \le \exp(c L^d).$$
\end{proposition}

\begin{proof}
 Without loss of generality, we can assume that $\supp(\varphi) \subset B(0,1)$, and hence $\supp(\varphi_L) \subset B(0,L)$ for every $L>0$. Let $\ind_{L}$ denote the indicator of the ball $B(0,L)$, we have
  $$|\Lambda(\varphi_L)| \le \L(\ind_{L}) \cdot \|\varphi\|_\infty.$$
Thus, it suffices to bound $\E[\exp(r \L(\ind_{L}))]$, for $r>0$ small enough.
 
Let $h_L(x):= 1 -\exp(r\ind_{L}(x)) = (1-e^r)\ind_{L}$ and denote by $M_{h_L}$ the multiplication operator by the function $h_L$ acting on $L(\mu)$. From Proposition \ref{prop:Laplace_det}, for $r>0$ small enough we have
$$  \E[\exp(r \L(\ind_{L}))] = \Det[I + \alpha M_{h_L} K]^{-1/\alpha}. $$
Let $\{\lambda_i\}_{i \ge 1}$ be the eigenvalues of the operator $\alpha K_{h_L}:= \alpha M_{h_L}K$. By choosing $r$ small enough, we may ensure that all the $\la_i$-s are smaller in modulus than $1/2$.
Using this, we may proceed as
\begin{align*}
|\E[\exp( r \L(\ind_{L}))]| =&  \exp\Big ( - \alpha^{-1} \log \Det [I + \alpha K_{h_L}] \Big ) 
=  \exp\Big ( - \alpha^{-1}  \sum_{i=1}^\infty\log (1 + \lambda_i) \Big ) \\ 
\le&   \exp\Big ( 2 |\alpha|^{-1}  \sum_{i=1}^\infty |\lambda_i| \Big )
\end{align*}
where we have used the fact that $0\le |\lambda_i| < 1/2$ and used  small parameter expansion for $\log (1+\la)$ for small $|\la|$.

We now investigate the $\lambda_i$-s, which are the eigenvalues of the operator $\alpha K_{h_L}$. For $r>0$, $h_L(x)$ is non-positive, which, coupled with the non-negativity of $K$, implies that the operator $K_{h_L}$ is negative semi-definite.  This implies in particular that all $\lambda_i$-s are of the same sign, hence
$$ \sum_{i=1}^\infty |\lambda_i| = |\Tr[\alpha K_{h_L}]|.$$
Since the operator  $\alpha K_{h_L}$ equals $\alpha(1-e^r)\ind_{L} K $, we have
$$\Tr[\alpha K_{h_L}] = \int \alpha K_{h_L}(x,x) \d \mu(x) =  \alpha(1-e^{r}) \int \ind_{L}(x) K(x,x) \d \mu(x) = \alpha ( 1 - e^r) \E[\L(\ind_{L})]. $$
Recall that $\E[\L(\ind_{L})] = \E[X(B(0,L))]\lesssim \Vol(B(0,L))$, thus
 $|\Tr[\alpha K_{h_L}]|  \lesssim L^d. $
Combining this with the argument above, we obtain $|\E[\exp(r\L(\ind_{L}))]|  \le \exp(cL^d)$ for $r>0$ small enough and for some $c>0$, as desired.
\end{proof}

\subsection{Proof of Theorem \ref{thm:CLT_det}}

\begin{proof}[Proof of Theorem \ref{thm:CLT_det}]

From Proposition \ref{prop:0-condition} and Proposition \ref{prop:mgf-growth} we deduce that there exist constants $r,c>0$ (not depending on $L$) such that  $\E[\exp(iu\Lambda (\varphi_L))]$ does not vanish in $\overline{\D(0,r)}$ and
$$\E[\exp(r| \Lambda (\varphi_L)|)] \leq \exp (cL^d) \text{ for all }L>0.$$

It remains to verify the growth condition in Corollary \ref{c:CLT}. From the inequality above, we have
$$\log^+\log \mathbb{E}[e^{r|\Lambda(\varphi_L)|}] \leq \log^+ \log \exp(cL^d) \leq \log^+c + d \log^+L.$$
Thus
$${1+ \log^+\log |\mathbb E[e^{r|\Lambda(\varphi_L)|}]| \over r\Var[\L(\varphi_L)]^{1/2}} \lesssim {\log L \over \Var[\L(\varphi_L)]^{1/2}} \overset{L\to\infty}{\longrightarrow} 0.$$
Using Corollary \ref{c:CLT}, the theorem follows.

\end{proof}

\subsection{Variance growth of $\Lambda(\varphi_L)$ for $\alpha>0$} \label{s:vargrowth}
Now we establish the variance growth condition for $\Lambda(\varphi_L)$, which is a necessary ingredient in our approach.
To this end, we recall that for any point process with one and two point correlation functions $\rho_1$ and $\rho_2$ respectively (with respect to background measure $\mu$) and any real-valued test function $\phi$, we have
\begin{equation}
\E[\Lambda(\phi)]=\int \phi(x) \rho_1(x) \d \mu(x)
\end{equation}
and
\begin{equation}
\E[\Lambda(\phi)^2]=\int \phi(x)^2  \rho_1(x) \d \mu(x)  + \iint \phi(x) \phi(y) \rho_2(x,y) \d \mu(x) \d \mu(y).
\end{equation}

For the $\alpha$-determinantal process, we have
$$\rho_1(x)=K(x,x) \quad , \quad \rho_2(x,y)=K(x,x)K(y,y)+\alpha |K(x,y)|^2.$$
Hence, by direct computation, we obtain
$$ \var[\Lambda(\varphi_L)] = \int \varphi_L(x)^2 K(x,x) \d \mu (x) + \alpha \iint \varphi_L(x) |K(x,y)|^2  \varphi_L(y)  \d \mu(x) \d \mu(y).$$
Now we are ready to state:

\begin{proposition} \label{prop:var-positive}
For $\alpha$-determinantal processes in Model \ref{model:det} (i) with $\alpha > 0$, for all $L>0$ large enough, we have
$$\var[\Lambda(\varphi_L)] \gtrsim L^d.$$
\end{proposition}

\begin{proof}
Since the kernel $K(\cdot ,\cdot)$ is positive semi-definite on $L^2(\mu)$, so is the kernel $|K(\cdot ,\cdot)|^2$. This follows from the so-called Schur product theorem for operators (e.g., see \cite{HJ}).
In particular,
$ \iint \varphi_L(x) |K(x,y)|^2  \varphi_L(y)  \d \mu(x) \d \mu(y) \ge 0.$
Hence, for $\alpha > 0$ we have
$$\var[\Lambda(\varphi_L)]  \ge \int \varphi_L(x)^2 K(x,x) \d \mu (x).$$
Since $\mu$ has a density $f$ with respect to the Lebesgue measure, which is assumed to be bounded from below by a positive constant, we may further write
$$ \var[\Lambda(\varphi_L)] \gtrsim \int \varphi_L(x)^2 K(x,x) \d x. $$

We recall the condition (i) in Model \ref{model:det}: for any $0<\delta<1$, there exist $m_\delta>0$ and $r>0$ such that
$$\Vol ( \mathcal{F}_{m_\delta} \cap C(x,r) ) \geq (1-\delta) \Vol (C(x,r)), \quad \forall x\in \R^d,$$
where $\mathcal{F}_{m_\delta}:= \{ y \in \R^d: K(y,y) \geq m_\delta \}$ and $C(x,r)$ is the cube centered at $x$ with side length $r$. From this assumption, we deduce that
\[\Vol ( \mathcal{F}_{m_\delta} \cap C(0,Nr) ) \geq (1-\delta) \Vol (C(0,Nr)), \quad \forall N\in \N_+.\]

Let $n_0\in \N_+$ be such that $\supp (\varphi) \subset C(0,n_0r)$. For $\varepsilon>0$, let 
$\Omega_{\varepsilon}:=\{x: |\varphi(x)| \geq \varepsilon\}.$
Since $\varphi$ is bounded and $\|\varphi \|_2 > 0$, there exists $\varepsilon >0 $ such that $\Vol(\Omega_{\varepsilon}) > 0$, and we define accordingly $\nu:= \Vol(\Omega_\varepsilon)/\Vol(C(0,n_0r))>0$. Since $L\cdot \Omega_\varepsilon \subset L \cdot \supp(\varphi) \subset C(0,Ln_0 r)$, we have
\begin{align*}
\Vol(\mathcal{F}_{m_\delta} \cap L \cdot \Omega_\varepsilon ) &= \Vol (\mathcal{F}_{m_\delta} \cap C(0,Ln_0r) \cap L \cdot \Omega_\varepsilon )\\
&\geq \Vol( \mathcal{F}_{m_\delta} \cap C(0,Ln_0r)) + \Vol(L \cdot \Omega_\varepsilon) - \Vol (C(0,Ln_0r))\\
&\geq \Vol( \mathcal{F}_{m_\delta} \cap C(0,{\lfloor L \rfloor}n_0r)) + \Vol(L \cdot \Omega_\varepsilon) - \Vol (C(0,Ln_0r))\\
&\geq (1-\delta)\Vol(C(0,{\lfloor L \rfloor}n_0r)) + \nu \Vol(C(0,Ln_0r)) - \Vol (C(0,Ln_0r))\\
&\geq \Big ((1-\delta)\Big ({L-1 \over L} \Big)^d + \nu - 1\Big ) \Vol(C(0,Ln_0r)).
\end{align*}
By choosing $\delta$ small enough, we deduce that $\Vol(\mathcal{F}_{m_\delta} \cap L \cdot \Omega_\varepsilon ) \gtrsim L^d$.
Thus,  for  $L$ large enough, we can further bound:
\begin{align*}
\int \varphi_L(x)^2 K(x,x) \d x &= \int \varphi (x/L)^2 K(x,x) \d x \\
 &\geq \int_{\mathcal{F}_{m_\delta} \cap L \cdot \Omega_\varepsilon} \varphi(x/L)^2 K(x,x) \d x \\
&\geq \varepsilon^2 m_\delta \cdot \Vol( \mathcal{F}_{m_\delta} \cap L \cdot \Omega_\varepsilon) \gtrsim L^d.
\end{align*}
This completes the proof.

\end{proof}

\begin{remark} \label{rem:weaker_cond}
While, for the sake of brevity, we establish Proposition \ref{prop:var-positive} as stated, we observe that the condition in Model \ref{model:det} (i) can be further weakened. In particular, . For that purpose, we can replace $m_\delta$ therein by $m_\delta \|x\|^{-L} \log^2 \|x\| c(\|x\|)$ for any function $c(\|x\|)$ going to infinity. To see this, we may use the same argument as in the proof of Proposition \ref{prop:var-positive}, but remove from $C(0,Ln_0r)$ a suitable cube of center $0$ and size $O(L)$. The estimate would then give $\log^2 L$ times a function $\to \infty$. It remains to observe that, in the context of the mgf growth bound in Proposition \ref{prop:mgf-growth}, we require a variance growth rate of only $\var[\L(\ph_L)]/\log^2 L \to \infty$ to deduce CLT via Theorem \ref{t:main-1}.
\end{remark}

\subsection{Variance growth of $\Lambda(\varphi_L)$ for $\alpha<0$}
Now we will discuss the case $\alpha<0$, which is more difficult and technical.
Recall that
\begin{equation} \label{eq:neg_var_start_0}
\var[\Lambda(\varphi_L)] = \int \varphi_L(x)^2 K(x,x) \d \mu (x) + \alpha \iint \varphi_L(x) |K(x,y)|^2  \varphi_L(y)  \d \mu(x) \d \mu(y).
\end{equation}
Since $\alpha<0$, we cannot simply neglect the second term, which is a source of difficulties in this setting.

\begin{lem} \label{lem:lb_variancedpp}
 Let $X$ be an $\alpha$-determinantal process with a Hermitian kernel $K$ and background measure $\mu$. Let $\varphi$ be any real-valued test function. Then
 \[ \var[\L(\varphi)] \ge (1 - |\alpha| \|K\|_\op) \int \varphi (x)^2 K(x,x) \d \mu(x) + \frac{|\alpha|}{2} \iint |\varphi(x)- \varphi(y)|^2 |K(x,y)|^2 \d\mu(y) \d\mu(x).\]
\end{lem}

\begin{proof}[Proof of Lemma \ref{lem:lb_variancedpp}]

We consider the operator composition $K^{\circ 2 } := K \circ K$, given by the integral kernel 
\[K^{\circ 2 }(x,y) = \int K(x,z)K(z,y) \d \mu(z).\] 
In particular, $K^{\circ 2 }(x,x)=\int |K(x,y)|^2 \d \mu(y)$ since $K$ is Hermitian. 
Moreover, 
\begin{equation} \label{eq:dom-1}
0 \preceq K^{\circ 2 } \preceq \|K\|_\op K, 
\end{equation}
where $\preceq$ denotes domination of operators in a positive semi-definite sense. 

For a bounded Borel set $E$, we denote by $\chi_E$ the characteristic function (i.e., the indicator function) of this set. Now we consider the inequality of local traces  (that is a consequence of \eqref{eq:dom-1}):
 \begin{equation} \label{eq:dom-2}
 0 \le \Tr [ \chi_E K^{\circ 2} \chi_E ] \le \|K\|_\op \cdot \Tr[ \chi_E K \chi_E ]
 \end{equation}
Writing out \eqref{eq:dom-2} as integrals, we obtain the following inequality, valid for any bounded Borel set $E$:
 \begin{equation} \label{eq:dom-3}
  \int_E \l( \int |K(x,y)|^2 \d \mu(y)  \r) \d \mu(x)  \le   \int_E \|K\|_\op \cdot  K(x,x) \d \mu(x).
 \end{equation}
 Since \eqref{eq:dom-3} is valid for every bounded Borel set $E$, we may deduce that 
 \begin{equation} \label{eq:dom-4}
  \int |K(x,y)|^2 \d \mu(y)  \le    \|K\|_\op \cdot  K(x,x)
 \end{equation}
holds for $\mu$-a.e. $x$.
Hence
\begin{align*}
& ~\int\varphi(x)^2 K(x,x) \d \mu (x)\\
=& ~(1-|\alpha|\|K\|_\op)\int\varphi(x)^2 K(x,x) \d \mu (x) + |\alpha|\|K\|_\op \int\varphi(x)^2 K(x,x) \d \mu (x) \\
 \ge & ~(1-|\alpha|\|K\|_\op)\int\varphi(x)^2 K(x,x) \d \mu (x) + |\alpha| \iint \varphi(x)^2 |K(x,y)|^2 \d \mu(y) \d \mu(x) \\
  = & ~(1-|\alpha|\|K\|_\op)\int\varphi(x)^2 K(x,x) \d \mu (x) + {|\alpha|\over 2} \iint (\varphi(x)^2 + \varphi(y)^2) |K(x,y)|^2 \d \mu(y) \d \mu(x).
\end{align*}

Using \eqref{eq:neg_var_start_0}, we can lower bound $\var[\L(\varphi)]$ by
\[
(1-|\alpha|\|K\|_\op)\int\varphi(x)^2 K(x,x) \d \mu (x) + {|\alpha|\over 2} \iint |\varphi(x)- \varphi(y)|^2 |K(x,y)|^2 \d \mu(y) \d \mu(x).
\]
\end{proof}

If $|\a|\|K\|_\op < 1$, Lemma \ref{lem:lb_variancedpp} implies the following lower bound
$$\var[\L(\varphi_L)] \gtrsim \int \varphi_L(x)^2 K(x,x) \d \mu (x).$$
With similar arguments as the case $\a>0$ we have the following result.

\begin{prop}
For $\alpha$-determinantal processes in Model \ref{model:det} (i) with $\alpha < 0$, for all $L>0$ large enough, we have
$$\var[\Lambda(\varphi_L)] \gtrsim L^d.$$
\end{prop}

However, for many kernel classes of interest, $|\a|\|K\|_\op = 1$ (for example, DPPs with projection kernels). In that case, the best lower bound for $\var[\L(\varphi_L)]$ we could obtain is
$$\var[\L(\varphi_L)] \gtrsim \iint |\varphi_L(x)- \varphi_L(y)|^2 |K(x,y)|^2 \d \mu(y) \d \mu(x).$$
Henceforth, we focus on lower bounding the double integral on the right hand side. 

By making a change of variables $u=x/L,v=y/L$, we may lower bound
\begin{align*}
&\iint |\varphi_L(x)-\varphi_L(y)|^2 |K(x,y)|^2 f(x) f(y) \d x \d y \\
 =  L^{2d} &\iint |\varphi(u)-\varphi(v)|^2 |K(Lu,Lv)|^2 f(Lu) f(Lv) \d u \d v \\
 \gtrsim   L^{2d} &\iint |\varphi(u)-\varphi(v)|^2 |K(Lu,Lv)|^2 \d u \d v,  \numberthis  \label{eq:int-1}
\end{align*}
where in the last step we use the hypothesis on the lower bound on the density $f$.

It now remains to lower bound the integral expression in \eqref{eq:int-1}. This will exploit the conditions (ii.a), (ii.b) or (ii.c) in Model  \ref{model:det}. The precise implementation of this programme will be taken up in Sections \ref{s:ii.a}, \ref{s:ii.b} and \ref{s:ii.c} respectively.

\subsubsection{Using hypothesis (ii.a) in Model \ref{model:det}} \label{s:ii.a}

\begin{prop}
For $\alpha$-determinantal processes in Model \ref{model:det} (ii.a), for all $L>0$ large enough, we have
$$\var[\Lambda(\varphi_L)] \gtrsim L^{d-4}.$$
\end{prop}

\begin{proof}
We recall the condition (ii.a) in Model 1.2: there exist positive constants $a,\delta$ such that
$|K(x,y)| \geq a$ whenever $\|x-y\|<\delta$.
Thus, we can further bound
\begin{equation} \label{eq:interm}
 L^{2d} \iint |\varphi(u)-\varphi(v)|^2 |K(Lu,Lv)|^2 \d u \d v \gtrsim L^{2d} \iint_{\|u-v\| < {\delta\over L}} |\varphi(u)-\varphi(v)|^2 \d u \d v.
\end{equation}

Let $B(x,r)$ be the ball with centre $x$ and radius $r$ in $\R^d$; for brevity we set $B=B(0,1)$. Let $\nu:= \pi^{d/2}/\Gamma (\frac{d}{2}+1)$ be the volume of $B$, and we denote by $\nu_L$ the volume of $B(0,\del/L)$. Thus, $\nu_L = (\delta/L)^d \nu$. 
We define the local average function $\bp$ by 
\[\bp(x):=\frac{1}{\nu_L} \int_{B(0,\del/L)} \varphi(x-u) \d u.\]

Set $\cb$ to the indicator function of $B$ and $\cbl$ to be that of $B(0,\del/L)$. We then have
\begin{align*}
&  \iint_{\|u-v\| < {\delta\over L}} |\varphi(u)-\varphi(v)|^2 \d u ~\d v \\
  = \:  & \nu_L \int_{\R^d} \Big ( \frac{1}{\nu_L} \int_{\|u-v\| < {\delta\over L}} |\varphi(u)-\varphi(v)|^2 \d v \Big ) \d u  \\
  \ge \: & \nu_L \int_{\R^d}  \Big |\varphi(u)-  \frac{1}{\nu_L} \int_{\|v-u\| < {\delta\over L}} \varphi(v) \d v \Big |^2   \d u  \quad (\text{by Jensen's inequality}) \\
= \: &    \nu_L \int_{\R^d}  \l|\varphi(u)-  \bp(u) \r|^2   \d u.  \numberthis \label{eq:interm-0}
\end{align*}

We note that $\bp =  \nu_L^{-1} \cdot \ph \ast \cbl$. By the Parseval-Plancherel Theorem, we have
\begin{equation} \label{eq:interm-1}
 \int_{\R^d}  \l|\varphi(u)-  \bp(u) \r|^2   \d u =  \|\varphi-  \bp\|_2^2   =   \|\hp-  \hat{\bp}\|_2^2 =  \|\hp-  \nu_L^{-1}\cdot \hp \cdot \widehat{\cbl}\|_2^2 = \|\hp \cdot (1-  \nu_L^{-1} \cdot \widehat{\cbl}) \|_2^2.
\end{equation}
Thus, we are reduced to examining the function $1-   \nu_L^{-1} \cdot \widehat{\cbl}$. In particular, we will show that it is $\gtrsim L^{-1}$ on a set that has a substantial intersection with $\supp(\hp)$.

To this end, we observe that $\cbl(x) = \cb\Big (\frac{L}{\delta} \cdot x \Big ) $. Via the behaviour of Fourier transforms under scaling (see \cite{Fi} for a complete derivation of the Fourier transform of the unit ball), we have
\begin{equation} \label{eq:Fourier_scaling-1}
\widehat{\cbl}(\xi)= \Big (\frac{\del}{L} \Big )^d \widehat{\cb} \Big (\frac{\del}{L} \cdot \xi \Big )= \Big (\frac{\del}{L}\Big )^d \cdot (2\pi)^{d/2} \Big \| {\del \over L} \cdot \xi \Big \|^{-d/2} J_{d/2} \Big ({\del \over L}\|\xi\| \Big ),
\end{equation}
where $J_{\b}$ is the Bessel function of the first kind with order $\b>0$.
We further have the following well-known power series expansion of such Bessel functions around $0$ (for details, see e.g. \cite{AS}) 
\begin{equation} \label{eq:Bessel}
J_{\b}(x) = \l({x \over 2}\r)^\b \sum_{k=0}^\infty \frac{(-1)^k x^{2k}}{4^k k! \Gamma(k+\b+1)} \cdot
\end{equation}
As a consequence, for small $x$, $J_\b$ can be further approximated as
\begin{equation} \label{eq:Bessel-1}
J_{\b}(x) = \l({x \over 2}\r)^\b  \cdot \l( \frac{1}{\Gamma(\b +1)} - \frac{x^2}{4\Gamma(\b+2)} + o(x^2)   \r).
\end{equation}

With the ingredients from \eqref{eq:Fourier_scaling-1}, \eqref{eq:Bessel-1}  in hand, we are ready to examine $\nu_L^{-1} \cdot \widehat{\cbl}(\xi)$ for $\xi$ in a fixed compact set (to be specified later). We proceed as
\begin{align*}
& \nu_L^{-1} \cdot \widehat{\cbl}(\xi) \\
= \: & \nu_L^{-1} \cdot \Big (\frac{\del}{L} \Big )^d \cdot   (2\pi)^{d/2}  \Big \| {\del \over L} \cdot \xi \Big \|^{-d/2} J_{d/2} \Big ({\del \over L}\|\xi\| \Big ) \\
= \: & \Big ( \Big ({\del \over L}\Big )^{d}\nu \Big )^{-1} \cdot  \Big (\frac{\del}{L}\Big )^{d/2}   \cdot   (2\pi)^{d/2}   \| \xi\|^{-d/2} \Big ({\del \|\xi\| \over 2L}\Big )^{d/2}  \cdot \Big ( \frac{1}{\Gamma(\frac{d}{2} +1)} - \frac{\del^2 \|\xi\|^2}{4L^2\Gamma(\frac{d}{2}+2)} + o(L^{-2})   \Big ) \\
= \: & \pi^{d/2} \nu^{-1} \cdot \Big ( \frac{1}{\Gamma(\frac{d}{2} +1)} - \frac{\del^2 \|\xi\|^2}{4L^2\Gamma(\frac{d}{2}+2)} + o(L^{-2})   \Big ) \\
= \: & \pi^{d/2} \Big (\frac{\pi^{d/2}}{\Gamma(\frac{d}{2}+1)}\Big )^{-1} \cdot \Big ( \frac{1}{\Gamma(\frac{d}{2} +1)} - \frac{\del^2 \|\xi\|^2}{4L^2\Gamma(\frac{d}{2}+2)} + o(L^{-2})   \Big ) \\
= \: & 1 - {\delta^2 \over 2 (d+2)} L^{-2} \|\xi\|^2 + o(L^{-2}) \\
= \: & 1 - \theta L^{-2} \|\xi\|^2 + o(L^{-2}),   \numberthis \label{eq:indicator_simplify} \quad \text{where $\theta := {\delta^2 \over 2 (d+2)} \cdot$}
\end{align*}

Since $\|\hp\|_2=\|\ph\|_2>0$, there exist $0<p<q<\infty$ such that \[\int_{\|\xi\| \in [p,q]} |\hp(\xi)|^2 \d \xi >0.\] 
Note that the set $\{p\le \|\xi\| \le q\}$ is compact, we now continue from \eqref{eq:interm-1} as
\begin{align*}
&  \|\hp \cdot (1-  \nu_L^{-1} \cdot \widehat{\cbl}) \|_2^2   \\
\ge \: & \int_{\|\xi\| \in [p,q]} |\hp(\xi)|^2 \cdot |1-  \nu_L^{-1} \cdot \widehat{\cbl}(\xi)|^2 \d \xi \\
= \: & \int_{\|\xi\| \in [p,q]} |\hp(\xi)|^2 \cdot |1-  (1 - \theta L^{-2} \|\xi\|^2 + o(L^{-2})) |^2 \d \xi  \quad (\text{using \eqref{eq:indicator_simplify} on $\{p\le \|\xi\| \le q\}$}) \\
\gtrsim \: & L^{-4} . \numberthis \label{eq:interm-2}
\end{align*}

Combining \eqref{eq:interm}, \eqref{eq:interm-0}, \eqref{eq:interm-1} and \eqref{eq:interm-2}, we may deduce that
\begin{equation}
\Var[\L(\varphi_L)] \gtrsim   L^{2d} \cdot \nu_L \cdot L^{-4} \gtrsim L^{d-4}.
\end{equation}

\end{proof}



\subsubsection{Using hypothesis (ii.b) in Model \ref{model:det}} \label{s:ii.b}

\begin{prop}
For $\alpha$-determinantal processes in Model \ref{model:det} (ii.b), for all $L>0$ large enough we have
$$\var[\Lambda(\varphi_L)] \gtrsim L^{2(d-\beta)}.$$
\end{prop}
\noindent
\textit{Proof.} We first recall that 
$$\var[\L(\varphi_L)] \gtrsim  L^{2d} \iint |\varphi(u)-\varphi(v)|^2 |K(Lu,Lv)|^2 \d u \d v.$$

Let $B:= \supp(\varphi)$, and let $ B_1$ be the set of points of distance at most $1$ to $B$ (hence, $B \subset B_1$). We may further lower bound
$$ L^{2d} \iint |\varphi(u)-\varphi(v)|^2 |K(Lu,Lv)|^2 \d u \d v
\ge  L^{2d} \int_{B} \int_{B_1^\complement} |\varphi(u)-\varphi(v)|^2 |K(Lu,Lv)|^2 \d u \d v. $$
Since $\varphi$ vanishes on $B_1^\complement$, we have
\begin{equation} \label{eq:int-2}
 L^{2d} \iint |\varphi(u)-\varphi(v)|^2 |K(Lu,Lv)|^2 \d u \d v \ge  L^{2d} \int_{B} \int_{B_1^\complement} |\varphi(u)|^2 |K(Lu,Lv)|^2 \d u \d v.
\end{equation}

We recall the condition (ii.b) in Model \ref{model:det}, which says that there exist positive constants $ \beta ,r, c_1,c_2$ with $d/2 < \beta < d$ and $n_0 \in \N_+$, such that for each $x \in \R^d$, the set
\[ E_x := \{y : |K(x,y)| \ge c_1 \|x-y\|^{-\beta} \} \]
satisfies 
\begin{equation} \label{eq:model_condii.b}
\Vol(E_x \cap A^x_n(r)) \ge c_2 \Vol(A^x_n(r)), \quad \forall n \ge n_0.
\end{equation}
\begin{lem} \label{Prop.2.5}
Let $K$ satisfy the above condition. Then there exist positive constants $c_3,R_0$ (independent of $x \in \R^d$), such that
\[ \int_{\|x-y\|>R} |K(x,y)|^2  \d y  \ge c_3 R^{d-2\beta}, \quad \forall R\ge R_0.  \]
\end{lem}

\begin{proof}[Proof of Lemma \ref{Prop.2.5}]
For simplicity of notations, we prove for the case $r=1$. The general case is similar. In what follows, we will denote $A_n^x := A_n^x(1)$.

Choosing $R_0=n_0$, then for any $R\geq R_0 $ and $x\in \mathbb R^d$, we have
$$\int_{\|x-y\|>R} |K(x,y)|^2 \d y
\geq \sum_{n=[R]+1}^{\infty} \int_{A_n^x} |K(x,y)|^2 \d y
\geq \sum_{n=[R]+1}^{\infty} \int_{A_n^x \cap E_x} |K(x,y)|^2 \d y.$$
Recall that $E_x = \{y: |K(x,y)| \ge c_1 \|x-y\|^\beta\}$ and $A_n^x= \{y: n \le \|x-y\| \le n+1\}$. Thus, on $A_n^x \cap E_x$, one has
\[ |K(x,y)|^2 \ge c_1^2 \|x-y\|^{-2\beta} \ge c_1^2 (n+1)^{-2\beta}. \]
This implies
$$\int_{A_n^x \cap E_x} |K(x,y)|^2 \d y \ge c_1^2 (n+1)^{-2\beta} \Vol (E_x \cap A_n^x) \ge c_1^2 c_2 (n+1)^{-2\beta} \Vol(A_n^x),$$
where we used \eqref{eq:model_condii.b} in the last inequality.
On $A_n^x$, one has $ \| x-y \|^{-2\beta} \le n^{-2\beta}$, which implies
\[\Vol (A_n^x) \ge n^{2\beta} \int_{A_n^x} \|x-y\|^{-2\beta} \d y.\]
Thus
$$\int_{A_n^x \cap E_x} |K(x,y)|^2 \d y \ge c_1^2 c_2 \Big ({n\over n+1}\Big )^{2\beta} \int_{A_n^x} \|x-y\|^{-2\beta} \d y \ge c_1^2 c_2 2^{-2\beta} \int_{A_n^x} \|x-y\|^{-2\beta} \d y.$$
Summing up over $n$ gives
\begin{eqnarray*}
\int_{\|x-y\|>R} |K(x,y)|^2 \d y &\ge& c_1^2 c_2 2^{-2\beta} \int_{\|x-y\|\ge [R]+1} \|x-y\|^{-2\beta} \d y \\
&\ge& c_1^2 c_2 2^{-2\beta} \int_{\|x-y\|\ge 2R} \|x-y\|^{-2\beta} \d y \\
&=& c_1^2 c_2 2^{-2\beta} R^{d-2\beta} \int_{\|u\| \geq 2} \|u\|^{-2\beta}\d u.
\end{eqnarray*}
Choosing $c_3:=c_1^2 c_2 2^{-2\beta} \int_{\|u\| \geq 2} \|u\|^{-2\beta}\d u$ completes the proof.
\end{proof}

We now continue from \eqref{eq:int-2}  with $L \ge L_0$ (for some $L_0$ large enough) as follows:
\begin{align*}
& L^{2d} \int_{B} \int_{B_1^\complement} |\varphi(u)|^2 |K(Lu,Lv)|^2 \d u \d v \\
= \: & L^d \int_{B} |\varphi(u)|^2  \Big ( L^d \int_{B_1^\complement} |K(Lu,Lv)|^2 \d v \Big )\d u \\
= \: & L^d \int_{B} |\varphi(u)|^2  \Big (  \int_{L \cdot B_1^\complement} |K(Lu,y)|^2 \d y \Big )\d u  \quad \text{(changing variables to $y=Lv$)} \\
\ge \: & c_3 L^d \int_{B} |\varphi(u)|^2  L^{d-2\beta} \d u \\
& \quad \quad \text{(using Lemma \ref{Prop.2.5} with $x=Lu$; note that $\|Lu-Lv\| \ge L$)} \\
= \: & c_3 \|\varphi\|_2^2 L^{2(d-\beta)}. \numberthis \label{eq:interm-3}
\end{align*}
Thus, we have the desired polynomial lower bound on the variance growth of $\Lambda(\varphi_L). ~\blacksquare$

\medskip

The following example shows that the Bessel kernels satisfy Model \ref{model:det} (ii.b):

\begin{example} \label{ex:Fourier}
The Bessel kernel $K(x,y) = \hat \chi_{B(0,1)} (x-y)$, where $\hat \chi_{B(0,1)}$ is the Fourier transform of the characteristic function of the unit ball ${B(0,1)}$ in $\R^d,  (d\geq 2)$, satisfies the decay condition in Model \ref{model:det} (ii.b) with
$$\beta = {d+1 \over 2}\quad,\quad r = 2\pi \quad,\quad c_1= {(2\pi)^{d/2} \over 4 \pi^{1/2}}\quad,\quad c_2 = {b-a \over 2^d \pi},$$
for some constants $a<b \in [0,2\pi]$.

Indeed, we have
$$K(x,y) = \hat\chi_{{B(0,1)}} (x-y) = (2\pi)^{d/2} \|x-y\|^{-d/2} J_{d/2}(\|x-y\|), $$
 where $J_n$ is the Bessel function of the first kind with order $n$ (\cite{Fi}). Use well-known asymptotics of Bessel functions (\cite{AS}), we have
$$ J_{d/2}(t) \sim \sqrt{2\over \pi t} \cos \Big ( t - {d+1\over 4} \pi \Big ) + O(t^{-3/2}) \quad \text{as $t\rightarrow +\infty$}.$$
Hence there exists a positive integer $n_0$ such that for all $t > n_0$, we have
$$ \sqrt{t} J_{d/2}(t) \geq \sqrt{2} \pi^{-1/2} \cos \Big ( t - {d+1\over 4} \pi \Big ) - {1\over 4} \pi^{-1/2}.$$
Let $[a,b] \subset [0,2\pi]$ be a subinterval such that $\cos \Big ( t - {d+1\over 4} \pi \Big ) \geq 1/2$ for all $t\in [a,b]$. Then for all positive integer $n>n_0$, we will have
$$\sqrt{t} J_{d/2}(t) \geq {\sqrt{2} \over 2} \pi^{-1/2} - {1\over 4}\pi^{-1/2} > {1\over 4} \pi^{-1/2}, \quad \forall t \in [2\pi n +a , 2\pi n +b].$$

For each $x\in \R^d$, consider the set
\begin{align*}
E_x:= &\Big \{y \in \R^d: |K(x,y)| \geq  {(2\pi)^{d/2} \over 4 \pi^{1/2}}\|x-y\|^{-(d+1)/2} \Big \} \\
=&\Big \{ y \in \R^d: \|x-y\|^{1/2} |J_{d/2}(\|x-y\|)| \geq {1\over 4} \pi^{-1/2} \Big \} \cdot
\end{align*}
By the argument above, we see that for every integer $n>n_0$
$$E_x \cap A_n^x(2\pi) \supset \{y \in \R^d: 2\pi n + a \leq \|x-y\| \leq 2\pi n + b\}.$$
Thus for all $n>n_0$
$${\Vol(E_x \cap A_n^x(2\pi) )\over \Vol(A_n^x(2\pi))} \geq {(2\pi n + b)^d - (2\pi n + a)^d \over (2\pi(n+1))^d - (2\pi n)^d} \geq {b-a \over 2^{d}\pi},$$
where we used the identity 
$x^d-y^d=(x-y)(x^{d-1}+x^{d-2}y+\cdots+y^{d-1}).$
\end{example}

\subsubsection{Using hypothesis (ii.c) in Model \ref{model:det}} \label{s:ii.c}

\begin{prop}
For $\alpha$-determinantal processes in Model \ref{model:det} (ii.c), for all $L>0$ large enough we have
$$\var[\Lambda(\varphi_L)] \gtrsim L^{d-2}.$$
\end{prop}

\begin{proof}
We first recall that 
\begin{eqnarray*}
\var[\L(\varphi_L)] &\gtrsim& \iint |\varphi_L(x)- \varphi_L(y)|^2 |K(x,y)|^2 \d \mu(y) \d \mu(x) \\
&\gtrsim& \iint |\varphi_L(x)- \varphi_L(y)|^2 |K(x,y)|^2 \d y \d x.
\end{eqnarray*}

Since $\varphi \in C^2_c$, for any $x,y \in \R^d$, we can write
$$\varphi(y) - \varphi(x) = \langle y-x, \nabla \varphi(x) \rangle + \langle y-x , \nabla^2 \varphi(\xi) \cdot (y-x) \rangle,$$
where $\nabla^2 \varphi(\xi)$ denotes the Hessian matrix of $\varphi$ at $\xi\in\R^d$. We further note that
\begin{eqnarray*}
|\langle y -x , \nabla^2 \varphi (\xi) \cdot (y-x) \rangle | \le \sup_{\xi\in\R^d} \|\nabla^2 \varphi(\xi)\|_\op \cdot \|y-x\|^2.
\end{eqnarray*}
Let $M:= \sup_{\xi \in \R^d} \| H \varphi(\xi)\|_\op$, which is finite since $\varphi \in C^2_c$. We then have 
$$|\varphi_L(y) - \varphi_L(x)| \ge {1\over L} \cdot |\langle y-x,\nabla \varphi(x/L) \rangle| - {M \over L^2} \cdot \|y-x\|^2,$$
which implies
\begin{eqnarray*}
|\varphi_L(y) - \varphi_L(x)|^2 &\ge& {1\over L^2} |\langle y-x, \nabla \varphi(x/L)\rangle |^2 - {2M \over L^3} |\langle y-x , \nabla \varphi (x/L) \rangle |\cdot \|y-x\|^2 \\
&\ge& {1\over L^2} |\langle y-x, \nabla \varphi(x/L)\rangle |^2 - {2M \over L^3} \|\nabla \varphi\|_\infty \cdot \|y-x\|^3,
\end{eqnarray*}
where $\|\nabla \varphi \|_\infty := \sup_{x\in\R^d}\|\nabla \varphi (x)\|$. 

Let $B$ be the closure of the set $\{x\in \R^d: \nabla \varphi (x) \neq 0 \}$ in $\R^d$. Since $\varphi$ is $C^2_c$, $B$ is compact. Then we can further lower bound
\begin{eqnarray*}
\var[\L(\varphi_L)] &\gtrsim& \int_{L\cdot B} \int_{\R^d}  |\varphi_L(x)- \varphi_L(y)|^2 |K(x,y)|^2 \d y \d x \\
&\gtrsim& {1\over L^2} \int_{L\cdot B} \int_{\R^d}  |\langle y-x, \nabla \varphi(x/L)\rangle |^2 |K(x,y)|^2 \d y \d x \\
&-& {2M\over L^3} \|\nabla \varphi \|_\infty \int_{L\cdot B} \int_{\R^d} \|y-x\|^3 |K(x,y)|^2 \d y \d x.
\end{eqnarray*}

Let $\kappa_1:= \inf_{x\in \R^d} \inf_{u\in \mathbb S^{d-1}} \int_{\R^d} |\langle y-x, u \rangle |^2 |K(x,y)|^2 \d y $, which is positive by the assumptions of Model \ref{model:det} (ii.c). Then 
$$\int_{\R^d}  |\langle y-x, \nabla \varphi(x/L)\rangle |^2 |K(x,y)|^2 \d y \ge \kappa_1 \| \nabla \varphi (x/L) \|^2 \quad \text{ for all } x \in \R^d,$$
which implies
$$ \int_{L\cdot B} \int_{\R^d}  |\varphi_L(x)- \varphi_L(y)|^2 |K(x,y)|^2 \d y \d x \ge \int_{L\cdot B} \kappa_1  \| \nabla \varphi (x/L) \|^2 \d x 
= \kappa_1 \|\nabla \varphi\|_2^2 \cdot L^d,$$
where $\|\nabla\varphi\|_2^2 := \int_{\R^d} \|\nabla \varphi (x)\|^2 \d x$.
Let $\kappa_2:= \sup_{x\in \R^d} \int_{\R^d} \|y-x\|^3 |K(x,y)|^2 \d y$, which is finite due to the assumptions of Model \ref{model:det} (ii.c). We then have
$$\int_{L\cdot B} \int_{\R^d} \|y-x\|^3 |K(x,y)|^2 \d y \d x \le \kappa_2 \Vol(L\cdot B) = \kappa_2 \Vol (B) \cdot L^d.$$
Combining every ingredients, we deduce that for $L>0$ large enough
$$\var[\L(\varphi_L)] \gtrsim L^{d-2}.$$
\end{proof}

Now we will discuss about the generality of condition \eqref{eq:rot_inv} in Model \ref{model:det} (ii.c). In particular, \eqref{eq:rot_inv} is satisfied for any translational and rotational invariant kernel $K$. 
Indeed, if $K$ is translational and rotational invariant, we can write $|K(x,y)|^2 = \Phi (\|x-y\|)$ for some univariate function $\Phi$. Let $u\in \mathbb S^{d-1}$ be any unit vector, we have
$$\int_{\R^d} |\langle  y-x , u \rangle |^2 |K(x,y)|^2 \d y = \int_{\R^d} |\langle  y-x , u \rangle |^2 \Phi(\|y-x\|) \d y = \int_{\R^d} |\langle v , u \rangle |^2 \Phi (\|v\|) \d v.$$
Let $\mathcal R\in \text{SO}(d)$ be any rotation, we can further write
$$\int_{\R^d} |\langle v,u\rangle |^2 \Phi(\|v\|) \d v 
= \int_{\R^d} |\langle \mathcal R v,u\rangle |^2 \Phi(\|v\|) \d v = \int_{\R^d} |\langle v, \mathcal R u\rangle |^2 \Phi(\|v\|) \d v.$$
Since $\text{SO}(d)$ acts transitively on $\mathbb S^{d-1}$, the value of $\int_{\R^d} |\langle  y-x , u \rangle |^2 |K(x,y)|^2 \d y$ does not depend on $x\in \R^d$ and $u\in \mathbb S^{d-1}$. By letting $u=e_i:=(0,\ldots,1,\ldots,0)^\top \in \R^d$ (the $1$ is at the $i$-th coordinate), for $i=1,\ldots,d$, we have
$$\inf_{x\in\R^d} \inf_{u\in \mathbb S^{d-1}} \int_{\R^d} |\langle  y-x , u \rangle |^2 |K(x,y)|^2 \d y = \int_{\R^d} |v_i|^2 \Phi(\|v\|)\d v = {1\over d} \int_{\R^d} \|v\|^2 \Phi (\|z\|) \d z.$$ 
If $\inf_{x\in\R^d} \inf_{u\in \mathbb S^{d-1}} \int_{\R^d} |\langle  y-x , u \rangle |^2 |K(x,y)|^2 \d y = 0$, we deduce that $\Phi (\|v\|) = 0$ almost everywhere. This is not possible, so the condition \eqref{eq:rot_inv} is satisfied.

\newpage

\begin{appendix}
\label{s:proof-1}
\section{A brief review on Marcinkiewicz theory}

Let $X$ be a real-valued random variable on a probability space.
We define the associated centered random variable by
$$\bar X:= X-\E[X].$$
Denote by $F_{X}:\R\to [0,1]$ the cumulative distribution function (c.d.f.) of $X$, that is,
$$F_{X}(x):=\P(X\leq x) \text{ \ for \ } x\in\R.$$
The c.d.f. of the Gaussian $N(0,1)$ is denoted by $\Phi$. Denote also by $\sigma^2$ the variance of $X$ with $\sigma\geq 0$.
Define $\log^+:=\max(\log,0)$. We will denote by  $\Psi_X(u):=\E[e^{iuX}]$ the characteristic function of $X$. For brevity, we will suppress the $X$-dependency from $\Psi_X$ and simply write $\Psi$ from time to time, whenever the random variable is understood from context.

The function $\Psi$, although well defined for any $u \in \R$, admits a holomorphic extension to the entire complex plane only under appropriate decay of the tails of the distribution of $X$. Entire characteristic functions form a function class of independent interest, and their distinctive properties such as growth rates and zero distributions have received considerable attention in the literature \cite{Pol,Ram,Luk,LSz}.

A fundamental result on entire characteristic functions is the classic  Marcinkiewicz Theorem \cite{Mar}. This entails that if $\Psi$ is an entire characteristic function that is of the form $\exp(P(u))$ for some polynomial $P$, then $X$ has to be a Gaussian (unless it is degenerate, i.e. purely atomic). The characteristic function $\Psi$ being of the form $\exp(f)$ for some entire function $f$ is equivalent to the assertion that $\Psi$ has no zeros on the whole of $\C$. For such functions, the growth rate, or equivalently, the order of the entire function $f$ is of considerable interest.

The simplest possible growth behaviour of $f$ arises when $X$ is degenerate, that is, a delta measure at a point; it is easy to see that in this case $f(u)$ grows at most linearly in $|u|$. It is well known that if $X$ follows a standard Poisson distribution, then $f(u)=c \cdot (e^{iu}-1)$ for some constant $c$. What growth rates for $f$ are possible in between these two extremities of linear and exponential growth is a fundamental question.

To set notations, we define $M_f(r):=\max\{|f(z)| : |z|=r \}$. In 1960, Linnik conjectured that the regime of Gaussianity, i.e. the regime of growth rate for which a Marcinkiewicz-type theorem holds true, extends all the way till $\log^+ M_f(r)=o(r)$, as $r \to \infty$. The most significant result in this direction is the theorem of Ostrovskii \cite{Os-a,Os-b,Os-1}, who demonstrates that Linnik's conjecture is indeed true as soon as $\limsup_{r \to \infty} r^{-1}\log^+ M_f(r) \newline = 0$, using ideas from Wiman-Valiron theory. By the very setup of these results, it is required that the entire characteristic function be of the form $\Psi(u)=\exp(f(u))$ for some entire function $f$; this entails in particular that the characteristic function $\Psi$ does not vanish anywhere on $\C$.

Theorem \ref{t:main-1} easily implies the following extension by Zimogljad \cite{Zim} of the classical Marcinkiewicz theorem:

\begin{corollary} \label{c:Marcin}
Let $X$ be a real-valued random variable such that $\E[e^{iuX}]=e^{f(u)}$ for some entire function $f(u)$.
Assume that
$$\liminf_{r\to\infty} r^{-1}  \log^+ \sup_{|u|=r} \Re f(u) =0.$$
Then, $f$ is a polynomial of degree at most equal to $2$.
\end{corollary}

In the classical Marcinkiewicz theorem, the function $f$ is assumed to be a polynomial and therefore the hypothesis on the growth of $f$ is automatically satisfied. Ostrovskii's theorem \cite{Os-a,Os-b,Os-1} entails a similar conclusion  as Corollary \ref{c:Marcin} with the stronger growth assumption
$$\limsup_{r\to\infty} r^{-1}  \log \sup_{|u|=r} |f(u)| =0.$$
This condition may be shown to be equivalent to
$$\limsup_{r\to\infty} r^{-1}  \log \sup_{|u|=r} \Re f(u) =0.$$

On a related note, we also refer to the preprint due to Eremenko and Fryntov \cite{EF}, which has been announced shortly prior to the announcement of our results, and also addresses the question of stability in Marcinkiewicz theorem, using a very different approach than ours. We emphasise that in our main theorems we only assume the zero freeness on a disk of finite radius. This is crucial for important applications, such as spin systems and general $\a$-determinantal processes examined in our paper (c.f. Remarks \ref{rem:spin}, \ref{rem:det} resp.), and is a main difference in comparison with the classical works by Marcinkiewicz, Ostrovskii, as well as \cite{Zim} and  \cite{EF}, where the zero freeness needs to hold for an infinite strip. In particular, a key technique using Phragm\`en-Lindel\"of principle for $\C$ or a strip in $\C$ in the approach of Ostrovskii, Zimoglyad and Eremenko-Fryntov doesn't work in our setting. Note also that we don't assume any hypothesis on the global behaviour of the characteristic function on $\C$ as needed in \cite{EF}. In particular, our characteristic functions do not need to exist on the whole $\C$.  

In this vein, we would like to mention a different flavour of stability results for the Marcinkiewicz theorem due to Golinskii (\cite{G-1}, see also \cite{GC}). This answers a question of Sapogov (\cite{Sap-1},\cite{Sap-2}) which asks for  conditions on the coefficients of a polynomial $P$ under the assumption that $\exp(P(\cdot))$ is \textit{close to} a characteristic function (but might not be a characteristic function itself).

It would be of interest to investigate if a combination of these approaches can lead to a more succinct and sharper quantitative Marcinkiewicz theory.

\section{Proofs of Theorem \ref{t:main-1} and Corollary \ref{c:CLT}}
We start by proving Theorem \ref{t:main-1}. 
In methodological terms, our proof strategy is similar spirit to \cite{MS-2}; however we obtain a simpler and more succinct  argument purely based on classical complex analytic techniques.
In particular, the following lemma may be seen as a version of \cite[Lemma 4.1]{MS-2} and is a key ingredient in the proof. 
Our proof of the lemma only uses Poisson formula and Dirichlet problem.
\begin{lemma} \label{l:key}
Let $h$ be a harmonic function on a neighbourhood of a rectangle $[-a,a]\times [0,b]$ such that $e^{\pi a/b}\geq 4\max|h|+1$ and $h(t)\geq h(t+is)$ for $t\in [-a,a]$ and $s\in [0,b]$.
Then, we have $h(is)-h(is')+1\geq 0$ for $0\leq s\leq s'\leq b-s$.
\end{lemma}
\proof
By reducing $b$, we can assume that $0\leq s\leq s'=b-s$.
Let $h_0, h_+, h_-$ be the solutions of the following Dirichlet problems on the rectangle $[-a,a]\times [0,b]$
$$\begin{cases}
\Delta h_0=0 \\
h_0=h \quad \text{on} \quad (-a,a)\times \{0,b\} \\
h_0=0 \quad \text{on} \quad \{-a,a\}\times (0,b)
\end{cases}
\quad 
\begin{cases}
\Delta h_\pm=0 \\
h_\pm=0 \quad \text{on} \quad (-a,a)\times \{0,b\} \text{ and } \{\mp a\}\times (0,b)\\
h_\pm=h \quad \text{on} \quad \{\pm a\}\times (0,b).
\end{cases}
$$
We have $h=h_0+h_++h_-$. So, 
it is enough to show that $h_0(is)\geq h_0(is')$ for $s, s'$ as above and $|h_\pm(is)|\leq 1/2$ for all $s\in (0,b)$.

We prove the first inequality.  Denote by $\D$ the unit disk in $\C$ and consider the unique conformal map 
$$\Pi: (-a,a)\times (0,b) \to \D \quad \text{such that} \quad  \Pi(ib/2)=0 \quad \text{and} \quad \Pi'(ib/2) \in\R_+.$$ 
It can be extended continuously to a bijective map $\Pi: [-a,a]\times [0,b] \to \overline \D$, see \cite[p.238]{StSh}. Write $\Pi = \Pi_2 \circ \Pi_1$, where $\Pi_1$ is the translation 
$u \mapsto u-ib/2$
and $\Pi_2$ is the unique conformal map from $(-a,a) \times (-b/2,b/2)$ to $\mathbb D$ with $\Pi_2(0)=0$ and $\Pi_2'(0) \in \R_+$. One can also extend $\Pi_2$ continuously to a bijective map $\Pi_2: [-a,a] \times [-b/2,b/2] \rightarrow \overline \D$.

Observe that the maps $z\mapsto \overline\Pi_2(\overline z), -\overline\Pi_2(-\overline z),-\Pi_2(-z)$ are conformal from $(-a,a) \times (-b/2,b/2)$ to $\mathbb D$ and satisfy the same properties at 0 as $\Pi_2$ does.
The uniqueness of $\Pi_2$ implies that all these maps are equal to $\Pi_2$. 
We easily deduce that $\Pi_2$ sends points in each half-line $e^{ik\pi/2}\R_+$, $k=0,1,2,3$, to the same half-line. In particular,
we have for some $y\in [0,1]$
$$\Pi(is) = \Pi_2(is - ib/2)=-iy \quad \text{and} \quad \Pi(is')= \Pi_2(-is+ib/2) = iy,$$
where we used $is' - ib/2 = - (is - ib/2)$ and $\Pi_2(z) = - \Pi_2(-z)$. Define $g(z):=h_0(\Pi^{-1}(z))$. For $0<\theta< \pi$, we have either
$$\Pi^{-1}(e^{-i\theta})\in (-a,a)\quad\text{and}\quad \Pi^{-1}(e^{i\theta})= \Pi^{-1}(e^{-i\theta})+ib$$
or these two points belong to the vertical edges of the rectangle where $h_0$ vanishes. It follows from the definition of $h_0$ that $g(e^{-i\theta})\geq g(e^{i\theta})$ for $0<\theta< \pi$.
We need to show that $g(-iy)\geq g(iy)$.

By Poisson's integral formula, we have 
\begin{eqnarray*} 
g(\pm iy) &=& {1\over 2\pi}\int_{-\pi}^\pi {1- y^2 \over |\pm iy - e^{i\theta}|^2} g(e^{i\theta}) d\theta \\
&=& {1\over 2\pi}\int_0^\pi {1- y^2 \over |\pm iy - e^{i\theta}|^2} g(e^{i\theta}) d\theta +{1\over 2\pi}\int_0^\pi {1- y^2 \over |\pm iy + e^{i\theta}|^2} g(e^{-i\theta}) d\theta.
\end{eqnarray*}
It follows that 
$$g(-iy)-g(iy)= {1\over 2\pi}\int_0^\pi(1-y^2)\Big[ {1\over |iy-e^{i\theta}|^2} -  {1\over |iy+e^{i\theta}|^2}\Big] \big(g(e^{-i\theta})-g(e^{i\theta})\big) d\theta.$$
It is clear that $g(-iy)-g(iy) \geq 0$ as each factor of the last integrand is non-negative.

It remains to prove the estimate for $h_\pm$. We only consider the case of $h_+$ as the case of $h_-$ can be obtained in the same way. Using a dilation of coordinate, we can assume for simplicity that $b=\pi$. We can solve explicitly the Dirichlet problem (see \cite[p.269]{AO}) and obtain
$$h_+(t+is)=\sum_{n=1}^\infty A_n \sinh(nt+na)\sin(ns) \quad \text{with} \quad A_n:={2\over \pi\sinh(2na)}\int_0^\pi h(a+i\xi) \sin(n\xi)d\xi.$$
Observe that $\sinh(t)/\sinh(2t) =1/(e^t+e^{-t}) \leq e^{-t}$ for $t\geq 0$. We then deduce that 
$$|h_+(is)|\leq 2\sum_{n=1}^\infty e^{-na} \max |h| = {2\over e^a-1} \max|h|\leq {1\over 2}\cdot $$
The lemma follows.
\endproof

We now prove a particular case of Theorem \ref{t:main-1}.

\begin{proposition} \label{p:main-1}
Let $X$ be as in Theorem \ref{t:main-1}. Assume moreover that
$X$ is centered and normalized, i.e., $\E[X]=0$ and $\sigma =1$. Let $\kappa:= \log^+ \log \E[e^{r|X|}]$. Then, we have for some universal constant $A>0$
$$\sup_{x\in \R} |F_X(x) - \Phi (x)| \leq {A(\kappa +1)\over r} \cdot$$
\end{proposition}

Observe that the left hand side of the above estimate is always bounded by 1 and we can choose $A$ large enough. Therefore, it is enough to consider 
 $r$ large enough. It follows that 
 $$\E[e^{r|X|}]\geq \E[r^2|X|^2/2]\geq r^2 \sigma^2/2=r^2/2$$ 
 and hence $\kappa$ is also large. 
 
By hypothesis, there is a function
$f(u)$ which is holomorphic on $\D(0,r)$ and continuous on $\overline{\D(0,r)}$ such that $f(0)=0$ and $e^{f(u)} = \E[e^{uX}]$. 
Consider its Taylor's expansion
$$f(u)=\sum_{n\geq 2} a_n u^n ={1\over 2} u^2+a_3 u^3+\cdots={1\over 2} u^2+R(u).$$
Since $X$ is real-valued, we have $a_n\in\R$ for every $n$.

Define $h(u):=\Re(f(u))=\log |\E[e^{uX}]|$. Since $X$ is real-valued, we have $h(u)=h(\overline u)$ for $|u|\leq r$. 
We then observe that
$h$ is  harmonic satisfying $h(0)=0$,  $h\geq 0$ on $[-r,r]$ (by Jensen's inequality)  and
$$h(u)  = \log |\E[e^{uX}]| \leq \log \E[e^{r|X|}] \leq  e^\kappa \quad \text{for} \quad |u|\leq r.$$
Thus, $\varphi(u):= e^{\kappa} - h(u)$ is a nonnegative harmonic function on $\D(0,r)$ and $\varphi(0) = e^{\kappa}$. By Harnack's inequality (see \cite[p.243]{Ah})
$$\varphi(u) \leq {r+|u|\over r-|u|} \varphi(0) \leq 5 \varphi(0) = 5e^{\kappa} \quad \text{for} \quad |u| \leq 2r/3,$$
which implies
$$|h(u)| \leq 4 e^{\kappa} \quad \text{for} \quad |u| \leq 2r/3.$$
Moreover, we have
\begin{equation} \label{e:h-ridge}
h(t) = \log |\E[e^{tX}]| \geq \log |\E[e^{tX}e^{isX}]|= h(t+is) \quad \text{for} \quad |t+is|\leq r.
\end{equation}

Define $r_1:= r/(2\kappa)$.

\begin{lemma} \label{l:b-decreasing}
We have for $|t|\leq r_1$ and $0\leq s\leq s'\leq r_1$
$$h(t+is)-h(t+is')+1\geq 0.$$
\end{lemma}
\proof
Fix $t\in [-r_1,r_1]$ and define 
$h_t(u):= h(t+u)$  for all $u\in [-r/2,r/2] \times [0,2r_1]$. 
Note that since $\kappa$ is large, we have for $u \in  [-r/2,r/2] \times [0,2r_1]$
$$|t+u|^2 \leq (r_1+r/2)^2 + 4r_1^2  < (2r/3)^2.$$
Thus, $h_t$ is a harmonic function on a neighborhood of $ [-r/2,r/2] \times [0,2r_1]$ and
$$ 4\max_{[-r/2,r/2] \times [0,2r_1]} |h_t| + 1 \leq 4 \sup_{\D(0,2r/3)} |h(u)| + 1 \leq 16e^{\kappa} + 1 \leq e^{\kappa+4}  \leq e^{\pi(r/2)/(2r_1)}. $$
Moreover, by \eqref{e:h-ridge}, we have for all $x\in [-r/2,r/2]$ and $y\in [0,2r_1]$
 $$h_t(x) = h(t+x) \geq h(t+x+iy) = h_t(x+iy).$$  
 Applying Lemma \ref{l:key} gives
$$h_t(is) - h_t(is') + 1 \geq 0 \quad \text{for} \quad 0 \leq s \leq s' \leq r_1,$$
or equivalently,
$$h(t+is) - h(t+is') + 1 \geq 0 \quad \text{for} \quad 0 \leq s \leq s' \leq r_1.$$
The lemma follows.
\endproof

Define $r_2:=r_1/3$.

\begin{lemma} \label{l:max}
There is a universal constant $c_0>0$ such that for every $0 \leq r'\leq r_2$ we have
$$|h(u)|\leq c_0\max(|h(ir')|,1)$$
for $u=t+is$ with $|t|\leq \sqrt{2} r'$ and $|s|\leq \sqrt{2} r'$.
\end{lemma}
\proof
We will define a number $M>0$ such that 
$|h(u)|\lesssim M$ and $|h(ir')|+1\gtrsim M$. The implicit constants we use in this lemma are universal.

Let $\Pi:(-3r'/2,3r'/2)^2 \to\D$ be the unique conformal map such that $\Pi(0)=0$ and  $\Pi'(0)\in\R_+$. It can be extended to a continuous bijective map
$\Pi:[-3r'/2,3r'/2]^2 \to \overline\D$. Define $g:=h\circ\Pi^{-1}$ and $g^\pm:=\max(\pm g,0)$.
Let $P(z,\theta)\geq 0$ denote the Poisson kernel on the unit disk $\D$ with $z\in\D$ and $\theta\in [0,2\pi]$. By Poisson's formula, we have 
\begin{eqnarray} \label{e:Poisson}
h(u)&=&g(\Pi(u)) \ = \ \int_0^{2\pi} P(\Pi(u),\theta) g(e^{i\theta}) d\theta \\
&=&\int_0^{2\pi} P(\Pi(u),\theta) g^+(e^{i\theta}) d\theta-\int_0^{2\pi} P(\Pi(u),\theta) g^-(e^{i\theta}) d\theta. \nonumber
\end{eqnarray}
Since $h(0)=0$ and $P(0,\theta)=1/(2\pi)$, we obtain 
$$\int_0^{2\pi} g^+(e^{i\theta}) d\theta=\int_0^{2\pi} g^-(e^{i\theta}) d\theta.$$

Define 
$$M := {1\over 2\pi}\int_0^{2\pi} |g(e^{i\theta})| d\theta  ={1\over \pi} \int_0^{2\pi} g^-(e^{i\theta}) d\theta \leq 2 \max_{[0,2\pi]} g^-(e^{i\theta}).$$
By the last inequality, we have $g(e^{i\theta})\leq -M/2$ for some $\theta \in [0,2\pi]$. It follows that 
 there is $\zeta_0$ in the boundary of $[-3r'/2,3r'/2]^2$ such that $h(\zeta_0)\leq -M/2$. Using the inequality in Lemma \ref{l:b-decreasing} and the fact that $h$ is symmetric, we find $\zeta_1\in [-3r'/2,3r'/2]\times \{3r'/2\}$ such that $h(\zeta_1)\leq -M/2+1$.

Now, consider two functions
$$\phi_1(u):=h(u)-h(u+ir')+1$$
and
$$\phi_2(u):=h(u)-h(\overline u+i3r'/2)+1.$$
Lemma \ref{l:b-decreasing} and the fact that $h$ is symmetric imply that the above functions are harmonic and non-negative on $(-3r',3r')\times (-r'/2,2r')$ and $(-3r',3r')\times (-3r'/2,3r'/4)$ respectively. 

Observe that $\phi_2(\Re(\zeta_1))\geq M/2$ because $h\geq 0$ on $[-r,r]$. By Harnack's inequality, we have $\phi_2(t+is)\gtrsim M$ for $|t|\leq 3r'/2$ and 
$-r'\leq s\leq r'/2$.  It follows that $\phi_1(ir'/4)=\phi_2(ir'/4)\gtrsim M$. By Harnack's inequality again, we have $\phi_1(0)\gtrsim M$.
Hence, $|h(ir')| +1 \gtrsim M$ because $h(0)=0$. 

On another hand, 
observe that if $K$ is a compact subset of $\D$, the above Poisson's formula \eqref{e:Poisson} implies that $\max_K|g|\leq c_K M$ for some constant $c_K$ depending only on $K$. We deduce from the definition of $g$ that $|h(u)|\lesssim M$ for  $u=t+is$ with $|t|\leq \sqrt{2} r'$ and $|s|\leq \sqrt{2} r'$. 
Now, the lemma follows easily.
\endproof

The following is a version of \cite[Lemma 6.1]{MS-2}. 

\begin{lemma} \label{l:large-index}
There exists a universal 
integer $N\geq 3$ such that for $1\leq r' \leq r_2$  we have
$$\sum_{n\geq N} |a_n|r'^n  \leq {1\over 300} \sum_{n=2}^{N-1}|a_n|r'^n .$$
\end{lemma}
\proof
By Lemma \ref{l:max}, the definition of $h$  and $a_2=1/2$, we have for $1\leq r'\leq r_2$
\begin{equation} \label{e:h-upper}
|h(u)|\leq c_0\max(|h(ir')|,1) \lesssim \sum_{n\geq 2} |a_n|r'^n
\end{equation}
when $|u|\leq \sqrt{2} r'$. 
Using a Cauchy's type formula and that $f(0)=0$, we also have
\begin{eqnarray*}
f(u) &=& i\Im f(0) +{1\over 2\pi} \int_0^{2\pi} {e^{i\theta}+ (\sqrt{2} r')^{-1} u \over e^{i\theta}- (\sqrt{2} r')^{-1} u} \Re f(\sqrt{2} r'e^{i\theta}) d\theta \\
&=& {1\over 2\pi} \int_0^{2\pi} \Big[{2\sqrt{2} r' e^{i\theta} \over \sqrt{2} r' e^{i\theta}-  u} -1\Big] h(\sqrt{2} r'e^{i\theta}) d\theta.
\end{eqnarray*}
By taking derivatives at 0 and using  \eqref{e:h-upper}, we obtain for $m\geq 2$
$$|a_m| =|f^{(m)}(0)/m!| \lesssim (\sqrt{2} r')^{-m} \max_{|u|=\sqrt{2}r'} |h(u)|\lesssim 2^{-m/2} r'^{-m} \sum_{n\geq 2} |a_n|r'^n.$$
Hence,
$$\sum_{n\geq N} |a_n|r'^n \lesssim 2^{-N/2}  \sum_{n\geq 2} |a_n|r'^n = 2^{-N/2} \sum_{n\geq N} |a_n|r'^n + 2^{-N/2} \sum_{n=2}^{N-1} |a_n|r'^n.$$
We easily obtain the lemma by taking $N$ large enough.
\endproof

We continue the proof of Proposition \ref{p:main-1} and study $a_n$ for $n<N$ using an improvement of Marcinkiewicz's argument. 
Consider the points $P_n=(n,\log(|a_n|r_2^n))$ in $\R^2$ for $2\leq n\leq N-1$.
Define also
$$P_{N}:=\big(N, \max_{2\leq n \leq N-1} \log(|a_n|r_2^n) \big).$$
This sequence of points admits a unique subsequence
$$P_{n_1}, \ldots, P_{n_{l-1}}, P_{n_l} \quad \text{with} \quad 2=n_1<\cdots<n_{l-1}<n_l=N,$$
such that

\begin{itemize}
\item If $\Gamma_j$ denotes the segment joining $P_{n_j}, P_{n_{j+1}}$ and $\Gamma$ is the union of the $\Gamma_j$, then
no point $P_n$,  with $2\leq n\leq N$, is above the polygon curve $\Gamma$;
\item If $\theta_j$ denotes the slope of $\Gamma_j$, then the sequence $(\theta_j)_{1\leq j\leq l-1}$ is strictly decreasing and $\theta_{l-1}=0$.
\end{itemize}

For example, if $\log(|a_n|r_2^n) \leq \log(|a_2|r_2^2)$ for $2\leq n\leq N-1$, then $l=2$ and $\Gamma$ is just a segment parallel to the abscissa axis of $\R^2$. Note that $\Gamma$ can be seen as the graph of a concave non-decreasing function over the interval $[2,N]$ which is constant on $[n_{l-1},N]$.

\begin{lemma} \label{l:Gamma}
We have
$\theta_{j-1}-\theta_j \leq 10$ for $2\leq j\leq l-1$ and $\theta_1\leq 10N$.
\end{lemma}

\proof
Since $\theta_{l-1}=0$, the second inequality is a direct consequence of the first one. We prove now the first inequality by contradiction. Recall that we have $ \Re f(t+is) \leq f(t)$ for $t,s\in\R$. 
The goal is to find $t,s$ such that $\Re f(t+is) > |f(t)|$, which contradicts the above inequality.

Let $2 \leq j \leq l-1$ be the largest integer for which the inequality in the lemma is wrong, that is
$$ \theta_{j-1} - \theta_j>10 \quad,\quad \theta_{k-1} - \theta_k \leq 10 \quad \text{for all} \quad j < k \leq l-1.$$
Since $\theta_{l-1}=0$, the maximality of $j$ implies that $\theta_j \leq 10(N-1)$. Define $\theta:=\theta_j+5$. We have
$$5 \leq \theta \leq 10N \quad \text{and} \quad \theta_{j-1} - \theta >5.$$

\noindent
{\bf Claim.} There is a complex number $u=t+is$ such that $t=\pm e^{-\theta}r_2$,  
$a_{n_j}\Re (u^{n_j})\geq \sqrt{3}|a_{n_j}| |u|^{n_j}/2$ and $2|t|/\sqrt{3}\leq |u|\leq 8|t|<r_2$.

\proof[Proof of Claim]
We will choose $t=\pm e^{-\theta}r_2$ and  $\arg(u)$ in $[\pi/6,11\pi/24]\cup [-11\pi/24, -\pi/6]$.
Define $\alpha:=\arg(u)$. Then $|u|/|t|=1/|\cos\alpha|$ and it is easy to check that $2|t|/\sqrt{3}\leq |u|\leq 8|t|<r_2$. 
When $\alpha$ runs over the above set and $n_j\not=3,4,6$, $\arg(u^{n_j})=n_j\alpha$ takes all possible values modulo $2\pi$.
So, the claim is clear in this case. When $n_j=3,4,6$, it is not difficult to check that there is a value of $\alpha$ in the considered set such that $\cos(n_j\alpha)\geq \sqrt{3}/2$ (resp. $\cos(n_j\alpha)\leq -\sqrt{3}/2$). Depending on the sign of $a_{n_j}$, the claim holds for one of these values of $\alpha$.
\endproof

We deduce from the properties of $\theta$ that $ e^{-10 N} r_2 \leq |t| \leq e^{-5} r_2$. 
Since we only consider big enough $r$, we can assume that $|t|\geq 1$. This allows us to apply Lemma \ref{l:large-index} for $t$ and $u$.

Denote by $L^-$ (respectively, $L^+$) the line through $P_{n_j}$ with slope $\theta_{j-1}$ (respectively, $\theta_j$). For any $2 \leq n \leq n_j-1$, $P_n$ is under the line $L^-$. This implies
$$ \log(|a_n|r_2^n) - \log(|a_{n_j}|r_2^{n_j}) \leq \theta_{j-1}(n - n_j)$$
hence
$$ |a_n|r_2^n \leq |a_{n_j}|r_2^{n_j} e^{\theta_{j-1}(n - n_j)}.$$
Multiplying both sides by $e^{-n\theta}$, we have 
$$ |a_n|(e^{-\theta}r_2)^n \leq |a_{n_j}|(e^{-\theta}r_2)^{n_j} e^{(\theta_{j-1}-\theta)(n - n_j)} \leq |a_{n_j}|(e^{-\theta}r_2)^{n_j} e^{5(n - n_j)}. $$
Hence,
\begin{equation} \label{eq:lm2.4.1}
|a_n||t|^n \leq |a_{n_j}||t|^{n_j}  e^{5(n - n_j)}, \quad 2\leq n \leq n_j-1.
\end{equation}
It follows that
$$\sum_{n=2}^{n_j-1} |a_n||t|^n \leq |a_{n_j}||t|^{n_j} \sum_{k=1}^{\infty} e^{-5k}=|a_{n_j}||t|^{n_j} {e^{-5}\over 1 -e^{-5}} < {1\over 100} |a_{n_j}||t|^{n_j}. $$

Similarly, for any $n_j+1 \leq n \leq N - 1$, we have that $P_n$ is under $L^+$. With the same argument, we obtain
\begin{equation} \label{eq:lm2.4.2}
|a_n||t|^n \leq |a_{n_j}||t|^{n_j}  e^{-5(n - n_j)}, \quad n_j + 1 \leq n \leq N - 1.
\end{equation}
Hence,
$$\sum_{n_j+1}^{N-1} |a_n||t|^n \leq |a_{n_j}||t|^{n_j} \sum_{k=1}^{\infty} e^{-5k}=|a_{n_j}||t|^{n_j} {e^{-5}\over 1 -e^{-5}} < {1\over 100} |a_{n_j}||t|^{n_j}. $$

Consider the following expansion of $f(t)$
$$f(t)=a_{n_j}t^{n_j} + \sum_{n=2}^{n_j-1} a_n t^n + \sum_{n=n_j+1}^{N-1} a_n t^n + \sum_{n=N}^{\infty} a_n t^n.$$
Denote by $S_1,S_2$ and $S_3$ the three summations in the above expansion of $f(t)$. From the above discussion, we deduce
$$|S_1|< \frac{1}{100} |a_{n_j}||t|^{n_j} \quad \text{and} \quad |S_2|< \frac{1}{100} |a_{n_j}||t|^{n_j} .$$
We apply Lemma \ref{l:large-index} for $t$ and obtain
$$|S_3|\leq {1\over 300} \Big(|a_{n_j}||t|^{n_j}+ |S_1|+|S_2|\Big)\leq {1\over 100} |a_{n_j}||t|^{n_j}.$$
Hence,
\begin{equation} \label{e:f(t)}
 |f(t)| \leq {103\over 100} |a_{n_j}||t|^{n_j}.
\end{equation}

Now, consider the following expansion of $f(u)$
\begin{equation} \label{e:f(u)}
f(u)=a_{n_j}u^{n_j} + \sum_{n=2}^{n_j-1} a_n u^n + \sum_{n=n_j+1}^{N-1} a_n u^n + \sum_{n=N}^{\infty} a_n u^n.
\end{equation}
Denote by $S_1',S_2'$ and $S_3'$ the three summations in this expansion. Using \eqref{eq:lm2.4.1} gives
\begin{align*}
 |S_1'| &\leq \sum_{n=2}^{n_j-1} |a_n| |u|^n = |a_{n_j}||u|^{n_j} \sum_{n=2}^{n_j-1} \frac{|a_n|}{|a_{n_j}|}|u|^{n-n_j}\\
 &\leq |a_{n_j}||u|^{n_j} \sum_{n=2}^{n_j-1} \frac{|a_n|}{|a_{n_j}|}|t|^{n-n_j}  \leq |a_{n_j}||u|^{n_j} \frac{e^{-5}}{1-e^{-5}} < \frac{1}{100}|a_{n_j}||u|^{n_j} .
\end{align*}
Using \eqref{eq:lm2.4.2} and the above Claim, we can bound  $S_2'$ as follows
\begin{align*}
 |S_2'| &\leq \sum_{n=n_j+1}^{N-1} |a_n| |u|^n = |a_{n_j}||u|^{n_j} \sum_{n=n_j+1}^{N-1} \frac{|a_n|}{|a_{n_j}|}|u|^{n-n_j}\\
 &\leq |a_{n_j}||u|^{n_j} \sum_{n=n_j+1}^{N-1} \frac{|a_n|}{|a_{n_j}|}(8|t|)^{n-n_j} \leq |a_{n_j}||u|^{n_j} \sum_{n=n_j+1}^{N-1} 8^{n-n_j} e^{-5(n-n_j)}  \\
 &\leq |a_{n_j}||u|^{n_j} \sum_{k=1}^{\infty} (8 e^{-5})^{k}  = |a_{n_j}||u|^{n_j} \frac{8 e^{-5}}{1-8 e^{-5}} \leq {8\over 100} |a_{n_j}||u|^{n_j} \cdot
\end{align*}
Applying Lemma \ref{l:large-index} gives 
$$|S'_3|\leq {1\over 300} \Big(|a_{n_j}||u|^{n_j}+ |S'_1|+|S'_2|\Big)\leq {1\over 100} |a_{n_j}||u|^{n_j}.$$
Hence, using \eqref{e:f(u)}, the Claim, together with \eqref{e:f(t)}, we get
\begin{align*}
 \Re f(u) &\geq \Re \big[a_{n_j} u^{n_j}\big] - {1\over 10}|a_{n_j}||u|^{n_j}\geq \Big({\sqrt{3}\over 2}-{1 \over 10}\Big)|a_{n_j}||u|^{n_j} \\
 &\geq \Big({\sqrt{3}\over 2}-{1 \over 10}\Big) \Big({2\over \sqrt{3}}\Big)^3 |a_{n_j}| |t|^{n_j} >{103\over 100} |a_{n_j}| |t|^{n_j} \geq |f(t)|.
 \end{align*}
This is the contradiction we are looking for. The proof is now completed.
\endproof

Define $r_3:=e^{-10N}r_2$. Recall that $f(u)=u^2/2+R(u)$. 

\begin{lemma} \label{l:R-est}
We have $|a_n| \leq {1\over 2} r_3^{2-n}$ for $2\leq n\leq N-1$. 
In particular, we have $|R(it)|\leq c_1 |t|^3r_3^{-1}$ for $|t|\leq r_3$, where $c_1>0$ is a universal constant.
\end{lemma}
\proof
By Lemma \ref{l:Gamma}, we have $\theta_1\leq 10N$. By the definition of $\theta_j$ and $r_3$, we deduce that $|a_n|r_3^n\leq |a_2| r_3^2$. We obtain the first inequality using that $a_2=1/2$. We prove now the second inequality.
By Lemma \ref{l:large-index} and the first assertion, we have
$$\sum_{N}^\infty |a_n| |t|^n \leq  |t/r_3|^N \sum_{N}^\infty |a_n| r_3^n \lesssim  |t/r_3|^N \sum_2^{N-1} |a_n| r_3^n\lesssim |t/r_3|^N r_3^2
\leq  |t|^3r_3^{-1}.$$
On the other hand, we have
$$\sum_3^{N-1} |a_n| |t|^n = \sum_3^{N-1} (|t|/r_3)^n |a_n| r_3^n \lesssim \sum_3^{N-1} (|t|/r_3)^n r_3^2  \lesssim (|t|/r_3)^3 r_3^2=|t|^3r_3^{-1}.$$
The result follows. 
\endproof

\proof[End of the proof of Proposition \ref{p:main-1}]
From the proof of the Berry-Esseen theorem \cite[p.538]{Fe}, for every positive number $T$, we have
$$\sup_{x\in \mathbb R} |F_X(x) - \Phi(x)| \lesssim  \int_{-T}^T \Big| {e^{f(it)} - e^{-t^2/2} \over t} \Big| dt + {1\over T} 
= \int_{-T}^T \Big| {e^{R(it)}-1\over t} \Big| e^{-t^2/2}dt + {1\over  T} \cdot $$
Choose $T=\delta r_3$ for some constant $\delta>0$ small enough.  We have
$$\sup_{x\in \mathbb R} |F_X(x) - \Phi(x)| \lesssim  \Big \{  \int_{|t| \leq 2\sqrt{\log r_3}} + \int_{|t|=2 \sqrt{\log r_3}}^{\delta r_3} \Big \} \Big| {e^{R(it)}-1\over t} \Big| e^{-t^2/2}dt + r_3^{-1} .$$

For the first integral, by Lemma \ref{l:R-est}, we have $|R(it)| \leq 1$ for $|t| \leq 2\sqrt{\log r_3}$ because $r_3$ is large. Using the fact $|e^R - 1| \leq e|R|$ for $|R| \leq 1$ and Lemma \ref{l:R-est} again, we deduce that
$$\int_{|t| \leq 2\sqrt{\log r_3}}  \Big| {e^{R(it)}-1\over t} \Big| e^{-t^2/2}dt \lesssim  \int_{|t| \leq 2\sqrt{\log r_3}} r_3^{-1} t^2 e^{-t^2/2} dt \lesssim r_3^{-1} \int_\R t^2 e^{-t^2/2}dt \lesssim r_3^{-1}.$$

For the second integral, observe that for $2\sqrt{\log r_3} \leq |t| \leq \delta r_3$ we have
$|R(it)| \leq t^2/4$ because $\delta$ is small.
Thus,
$$  \int_{|t|=2 \sqrt{\log r_3}}^{\delta r_3} \Big| {e^{R(it)}-1\over t} \Big| e^{-t^2/2}dt \lesssim  \int_{|t|=2\sqrt{\log r_3}}^{\delta r_3} e^{t^2/4} e^{-t^2/2}dt \lesssim\int_{2\sqrt{\log r_3}}^\infty te^{-t^2/4}dt \lesssim r_3^{-1}.$$

Combining everything, we have
$$\sup_{x\in \mathbb R} |F_X(x) - \Phi(x)| \lesssim r_3^{-1}\lesssim {\kappa \over r} \cdot $$
This ends the proof of the proposition.
\endproof

\proof[End of the proof of Theorem \ref{t:main-1}]
We can assume that $|\sigma-1|<1/2$ and hence $\sigma>1/2$  because the theorem is clear otherwise. 
Define 
$$\kappa := \log^+\log \E[e^{r|X|}] \leq \log^+\log (\E[e^{rX}]+\E[e^{-rX}])\leq 1+ \log^+\log \max_{|u|=r} |\E[e^{iuX}]|.$$ 
Define also
$$\hat X := {X - \E[X] \over \sigma } = {\bar X \over \sigma} \quad,\quad \hat r := {r\over 2} < r\sigma \quad,\quad \hat \kappa := \log^+ \log \E[e^{\hat r|\hat X|}].$$
By Jensen's inequality, we have
$$\E[r|X|] \leq \log \E[e^{r|X|}] \leq e^{\kappa}$$
and hence
$$\log \E[e^{\hat r |\hat X|}] \leq \log \E[e^{r\sigma |\hat X|}] = \log \E[e^{r|X - \E[X]|}] \leq \log \E[e^{r|X|}] + r\E[|X|] \leq 2e^{\kappa} < e^{\kappa+1}.$$
This implies $\hat \kappa < \kappa +1$. Applying Proposition \ref{p:main-1} for $\hat X, \hat r, \hat \kappa $ gives
$$\sup_{x\in \mathbb R} |F_{\hat X}(x) - \Phi(x)| \lesssim {\hat \kappa \over \hat r} \lesssim {\kappa+1\over r} \lesssim 
{1+ \log^+\log \max_{|u|=r} |\E[e^{iuX}]| \over r}\cdot $$
On another hand, we have 
\begin{eqnarray*}
\sup_{x\in\R}|F_{\bar X}(x)-\Phi(x)| &=& \sup_{x\in\R} |F_{\hat X}(x)-\Phi(\sigma x)| \\
&\leq&  \sup_{x\in\R} |F_{\hat X}(x)-\Phi(x)| + \sup_{x\in\R} |\Phi(\sigma x)-\Phi(x)|.
\end{eqnarray*}

So, in order to get the desired estimate, we only need to bound $|\Phi(\sigma x)-\Phi(x)|$ by $2|\sigma-1|$ for $x\in\R$.
Observe that since $|\sigma-1|\leq 1/2$, we have 
\begin{eqnarray*}
|\Phi(\sigma x)-\Phi(x)| &=& {1\over \sqrt{2\pi}} \Big|\int_{\sigma x}^x e^{-t^2/2} dt\Big| \ = \  {1\over \sqrt{2\pi}} \Big|\int_\sigma^1 xe^{-x^2s^2/2} ds \Big| \\
&\leq& {1\over \sqrt{2\pi}} |\sigma-1| |x| e^{-x^2/8}\  \leq \ 2|\sigma-1|
\end{eqnarray*}
because $|x|e^{-x^2/8}$ attains its maximum when $|x|=4$. 
The proof is finished.
\endproof


\proof[Proof of Corollary \ref{c:CLT}]
Without loss of generality, we can replace $X_n, r_n$ by $X_n/\sigma_n, r_n\sigma_n$ so that we can assume that $\sigma_n=1$. Now, it is enough to apply Theorem \ref{t:main-1} to $X_n,r_n, 1$ instead of $X,r, \sigma$. 
\endproof

\end{appendix}

\begin{acks}[Acknowledgments]
T.C.D. was supported by the NUS and MOE grants  R-146-000-319-114 and MOE-T2EP20120-0010. S.G. was supported in part by the MOE grants  R-146-000-250-133, R-146-000-312-114 and MOE-T2EP20121-0013. 
H.S.T. was supported by the NUS Research
Scholarship.
We would like to thank Alexandre Eremenko for  bringing to our notice  the work of L.B. Golinskii. We would also like to thank Julian Sahasrabudhe and Marcus Michelen for illuminating discussions, and for pointing us to Theorem 12.2  in \cite{MS-2}. We additionally thank Frank den Hollander for his feedback and thoughtful discussions. 
\end{acks}


\begin{thebibliography}{99}
\bibitem[AbStRo]{AS}
M. Abramowitz, I.A. Stegun, R.H. Romer, \emph{Handbook of mathematical functions with formulas, graphs, and mathematical tables}. (1988): 958-958.

\bibitem[AgO]{AO}
R.P. Agarwal, D.  O'Regan, \emph{Ordinary and partial differential equations.
With special functions, Fourier series, and boundary value problems.}
Universitext. Springer, New York,  2009. xiv+410 pp.

\bibitem[Ah]{Ah}
L.V. Ahlfors,  \emph{Complex analysis.
An introduction to the theory of analytic functions of one complex variable.}
Third edition. International Series in Pure and Applied Mathematics. McGraw-Hill Book Co., New York,  1978. {\rm xi}+331 pp.

\bibitem[BaBlKa]{Bac}
F. Baccelli,  B. Blaszczyszyn,  M. Karray, \emph{Random measures, point processes, and stochastic geometry}. (2020). https://hal.inria.fr/hal-02460214

\bibitem[BaHa]{Bar}
R. Bardenet, A. Hardy, \emph{Monte Carlo with determinantal point processes}. The Annals of Applied Probability 30, no. 1 (2020): 368-417.

\bibitem[Ba]{Ba}
R. Bauerschmidt, \emph{Ferromagnetic spin systems}. Lecture notes (2016).

\bibitem[Ber]{Berry}
A.C. Berry,  \emph{The accuracy of the Gaussian approximation to the sum of independent variates}. Transactions of the american mathematical society 49, no. 1 (1941): 122-136.

\bibitem[Bor]{Bor}
A. Borodin,  \emph{Determinantal point processes.} In The Oxford Handbook of Random Matrix Theory (2018).

\bibitem[BrDu]{Duits}
J. Breuer, M. Duits, \emph{Central limit theorems for biorthogonal ensembles and asymptotics of recurrence coefficients}. Journal of the American Mathematical Society 30, no. 1 (2017): 27-66.

\bibitem[CuMaOc]{Maj}
F. Cunden, S.N. Majumdar, N. O'Connell, \emph{Free fermions and $\alpha$-determinantal processes}. Journal of Physics A: Mathematical and Theoretical 52, no. 16 (2019): 165202. 21pp.

\bibitem[DaVe]{DV}
D. Daley, D. Vere-Jones, \emph{An introduction to the theory of point processes}.  Vols. I \& II, Springer New York, 2003.

\bibitem[Dur]{Dur}
R. Durrett, \emph{Probability: theory and examples}. Vol. 49. Cambridge university press, 2019.

\bibitem[Ell]{Ellis_book}
R. Ellis,  \emph{Entropy, large deviations, and statistical mechanics}. Vol. 1431, no. 821. Taylor \& Francis, 2006.

\bibitem[ErFr]{EF}
A. Eremenko, A. Fryntov, \emph{Stability in the Marcinkiewicz theorem}. arXiv preprint, arXiv:2106.14078 (2021).

\bibitem[Es42]{Es-1}
C.-G. Esseen, \emph{On the Liapunoff limit of error in the theory of probability}. Arkiv for matematik, astronomi och fysik, A: 1-19 (1942).

\bibitem[Es56]{Es-2}
C.-G. Esseen,  \emph{A moment inequality with an application to the central limit theorem}. Scandinavian Actuarial Journal 1956, no. 2 (1956): 160-170.

\bibitem[Fe]{Fe}
W. Feller, \emph{An introduction to probability theory and its applications}. Vol. II. Second edition John Wiley \& Sons, Inc., New York-London-Sydney 1971 xxiv+669 pp.

\bibitem[Fi]{Fi}
D. Fischer, \emph{Fourier transform of the indicator of the unit ball}. URL (version: 2013-09-12): https://math.stackexchange.com/q/492055

\bibitem[FrRo]{FR}
J. Fr\"ohlich, P.-F. Rodriguez, \emph{Some applications of the Lee-Yang theorem}. Journal of mathematical physics 53, no. 9 (2012): 095218.

\bibitem[Gol]{G-1}
L.B. Golinskii, \emph{Stability estimates in the theorem of J. Marcinkiewicz}, Journal of Soviet Mathematics 57, no. 4 (1991): 3193-3209.

\bibitem[GoCh]{GC}
L.B. Golinskii, G.P. Chistyakov, \emph{On Stability Estimates in the Marcinkiewicz Theorem and Its Generalization}, Entire and Subharmonic Functions 11 (1992): 243.

\bibitem[GhLiPe]{GLP}
S. Ghosh, T. M. Liggett, R. Pemantle, \emph{Multivariate CLT follows from strong Rayleigh property}. In 2017 Proceedings of the Fourteenth Workshop on Analytic Algorithmics and Combinatorics (ANALCO), pp. 139-147. Society for Industrial and Applied Mathematics, 2017.

\bibitem[Gin]{Gin}
J. Ginibre,  \emph{Rigorous lower bound on the compressibility of a classical system}. Physics Letters A 24, no. 4 (1967): 223-224.

\bibitem[HoJo]{HJ}
R.A. Horn,  C.R. Johnson, \emph{Matrix analysis}. Cambridge university press, 2012.

\bibitem[HKPV]{HKPV}
J.B. Hough, M. Krishnapur, Y. Peres, B. Vir\'ag, \emph{Determinantal processes and independence}. Probability surveys 3 (2006): 206-229.


\bibitem[IaSo]{IS}
D. Iagolnitzer, B. Souillard, \emph{Lee-Yang theory and normal fluctuations}. Physical Review B, Vol. 19,
  No. 3 (1979): 1515.

\bibitem[Ka]{Ka}
O. Kallenberg, \emph{Foundations of modern probability}. Vols. 1 \& 2. New York: springer, 1997.

\bibitem[KuPfVy]{Cor-1}
H. Kunz,  Ch.-Ed. Pfister, P.-A. Vuillermot, \emph{Correlation inequalities for some classical spin vector models}. Physics Letters A 54, no. 6 (1975): 428-430.

\bibitem[LPRS]{LPRS}
J.L. Lebowitz, B. Pittel, D. Ruelle, E.R. Speer, \emph{Central limit theorems, Lee-Yang zeros, and graph-counting polynomials}. Journal of Combinatorial Theory, Series A 141 (2016): 147-183.

\bibitem[LY-I]{YL-1}
C.-N. Yang, T.-D. Lee, \emph{Statistical theory of equations of state and phase transitions. I. Theory of condensation}. Physical Review 87, no. 3 (1952): 404.

\bibitem[LY-II]{LY-2}
T.-D. Lee, C.-N. Yang, \emph{Statistical theory of equations of state and phase transitions. II. Lattice gas and Ising model}. Physical Review 87, no. 3 (1952): 410.

\bibitem[LiSo]{LS}
E. Lieb, A. Sokal, \emph{A general Lee-Yang theorem for one-component and multicomponent ferromagnets}. Communications in Mathematical Physics 80, no. 2 (1981): 153-179.

\bibitem[LiOs]{LiOs}
Yu.V. Linnik, I.V. Ostrovskii, \emph{Decomposition of random variables and vectors}. Translated from the Russian. Translations of Mathematical Monographs, Vol. 48. American Mathematical Society, Providence, R. I., 1977. ix+380 pp.

\bibitem[Luk58]{Luk-1}
E. Lukacs, \emph{Some extensions of a theorem of Marcinkiewicz}. Pacific Journal of Mathematics 8, no. 3 (1958): 487-501.

\bibitem[Luk72]{Luk}
E. Lukacs, \emph{A survey of the theory of characteristic functions}. Advances in Applied Probability 4, no. 1 (1972): 1-37.

\bibitem[LuSz]{LSz}
E. Lukacs, O. Sz\'asz, \emph{On analytic characteristic functions}.  Pacific J. Math 2, no. 4 (1952): 615-625.

\bibitem[Mar]{Mar}
J. Marcinkiewicz, \emph{Sur une propriété de la loi de Gauss}. Mathematische Zeitschrift 44, no. 1 (1939): 612-618.

\bibitem[MS-I]{MS-1}
M. Michelen, J. Sahasrabudhe, \emph{Central limit theorems from the roots of probability generating functions}. Advances in Mathematics 358 (2019): 106840.

\bibitem[MS-II]{MS-2}
M. Michelen, J. Sahasrabudhe, \emph{Central limit theorems and the geometry of polynomials}. arXiv preprint, arXiv:1908.09020 (2019).

\bibitem[MoPi]{Cor-2}
J.L. Monroe, P.A. Pearce, \emph{Correlation inequalities for vector spin models}. J. Stat. Phys. 21, 615-633 (1979).

\bibitem[New]{New}
C. Newman, \emph{Zeros of the partition function for generalized Ising systems}. Communications on Pure and Applied Mathematics 27, no. 2 (1974): 143-159.

\bibitem[Os62]{Os-b}
I.V. Ostrovskii,  \emph{On applications of one relationship due to Wiman and Valiron}. Doklady AN SSSR, vol. 143, no. 3, pp. 532-535. 1962.

\bibitem[Os63]{Os-a}
I.V. Ostrovskii,  \emph{On entire functions satisfying some special inequalities connected with the theory of characteristic functions of probability laws}. Har'kov. Univ. Mat. Obshestva 29 (1963): 145-168.

\bibitem[Os83]{Os-1}
I.V. Ostrovskii,  \emph{On the growth of entire characteristic functions}. In Stability Problems for Stochastic Models, pp. 151-155. Springer, Berlin, Heidelberg, 1983.

\bibitem[Os84]{Os-2}
I.V. Ostrovskii, \emph{On the growth of ridge functions that are entire or analytic in a half-plane}. Mathematics of the USSR-Sbornik 47, no. 1 (1984): 145.

\bibitem[Pol]{Pol}
G. P\'olya, \emph{Remarks on characteristic functions}. In Proceedings of the [First] Berkeley Symposium on Mathematical Statistics and Probability (pp. 115-123). University of California Press. (1949, January).

\bibitem[Ram]{Ram}
B. Ramachandran,  \emph{On the order and the type of entire characteristic functions}. The Annals of Mathematical Statistics (1962): 1238-1255.

\bibitem[RiVi]{RV}
B. Rider, B. Vir\'ag, \emph{The noise in the circular law and the Gaussian free field}. International Mathematics Research Notices (2007):  no. 2, Art. ID rnm006, 33 pp.


\bibitem[Sap-I]{Sap-1}
N.A. Sapogov, \emph{Stability for the Marcinkiewicz theorem. Case of a fourth-degree polynomial}, Zapiski Nauchnykh Seminarov POMI 61 (1976): 107-124.

\bibitem[Sap-II]{Sap-2}
N.A. Sapogov, \emph{The problem of stability for J. Marcinkiewicz's theorem}, Zapiski Nauchnykh Seminarov POMI 87 (1979): 104-124.


\bibitem[ShTa]{ST}
T. Shirai, Y. Takahashi, \emph{Random point fields associated with certain Fredholm determinants I: fermion, Poisson and boson point processes}. Journal of Functional Analysis 205, no. 2 (2003): 414-463.

\bibitem[SiGr]{SG}
B. Simon, R. Griffiths, \emph{The $(\phi_4)^2$ field theory as a classical Ising model}. Communications in Mathematical Physics 33, no. 2 (1973): 145-164.

\bibitem[Sos-I]{Sos}
A. Soshnikov, \emph{Determinantal random point fields}. Russian Mathematical Surveys 55, no. 5 (2000): 923.

\bibitem[Sos-II]{Sos-1}
A. Soshnikov,  \emph{The central limit theorem for local linear statistics in classical compact groups and related combinatorial identities}. Annals of probability (2000): 1353-1370.

\bibitem[StSh]{StSh}
E.M. Stein, R. Shakarchi,  \emph{Complex analysis.}
Princeton Lectures in Analysis, 2. Princeton University Press, Princeton, NJ,  2003. xviii+379 pp. 


\bibitem[DoTi]{Tir-1}
R. L. Dobrushin, B. Tirozzi, \emph{The central limit theorem and the problem of equivalence of ensembles}. Communications in Mathematical Physics 54, no. 2 (1977): 173-192.


\bibitem[CaDeTi]{Tir-2}
M. Campanino, G. D. Grosso, B. Tirozzi, \emph{Local limit theorem for Gibbs random fields of particles and unbounded spins.} Journal of Mathematical Physics 20, no. 8 (1979): 1752-1758.

\bibitem[Zim]{Zim}
V. Zimogljad, \emph{The growth of entire functions that satisfy special inequalities}. Teor. Funkcii Funkcional. Anal. i Prilozen., 6 (1968) 30-41.


\bibitem[CoLe]{CL}
O. Costin, J. L. Lebowitz, \emph{Gaussian fluctuation in random matrices.} Physical Review Letters 75, no. 1 (1995): 69.

\bibitem[Leb]{Leb}
J. L. Lebowitz, \emph{Charge fluctuations in Coulomb systems.} Physical Review A 27, no. 3 (1983): 1491.

\bibitem[Sos]{Sos}
A. Soshnikov, \emph{The central limit theorem for local linear statistics in classical compact groups and related combinatorial identities.} Annals of probability (2000): 1353-1370.

\bibitem[GhLe]{GL-1}
S. Ghosh, J. L. Lebowitz, \emph{Fluctuations, large deviations and rigidity in hyperuniform systems: a brief survey.} Indian Journal of Pure and Applied Mathematics 48, no. 4 (2017): 609-631.

\bibitem[AdGhLe]{GL-2}
K. Adhikari, S. Ghosh, J. L. Lebowitz, \emph{Fluctuation and entropy in spectrally constrained random fields.} Communications in Mathematical Physics 386, no. 2 (2021): 749-780.

\bibitem[LS]{Lam-Sen}Lam, W. \& Sen, A. Central limit theorem in disordered Monomer-Dimer model. {\em Random Structures \& Algorithms}. \textbf{66}, e21256 (2025)

\bibitem[Cha19]{Cha3}Chatterjee, S. Central limit theorem for the free energy of the random field Ising model. {\em Journal Of Statistical Physics}. \textbf{175}, 185-202 (2019)

\bibitem[Cha07]{Cha1}Chatterjee, S. Stein’s method for concentration inequalities. {\em Probability Theory And Related Fields}. \textbf{138} pp. 305-321 (2007)

\bibitem[Cha08]{Cha2}Chatterjee, S. A new method of normal approximation. {\em Ann. Probab.}. \textbf{36}, 1584-1610 (2008)


\end{thebibliography}
\end{document}